\documentclass[11pt,reqno]{amsart}

\usepackage{amssymb,amsmath,amsfonts,amsthm}
\usepackage{bm,cite,graphicx, color}
\usepackage[toc,page]{appendix}

\newtheorem{theorem}{Theorem}[section]
\newtheorem{lemma}[theorem]{Lemma}
\newtheorem{corollary}[theorem]{Corollary}
\newtheorem{definition}[theorem]{Definition}

\newtheorem{assumption}{Assumption}
\numberwithin{equation}{section}

\allowdisplaybreaks[4]

\begin{document}

\title[inverse elastic scattering for a random potential]{Inverse Elastic
Scattering for a Random Potential}

\author{Jianliang Li}
\address{School of Mathematics and Statistics, Changsha University of Science
and Technology, Changsha 410114, P.R. China.}
\email{lijl@amss.ac.cn}

\author{Peijun Li}
\address{Department of Mathematics, Purdue University, West Lafayette, Indiana
47907, USA.}
\email{lipeijun@math.purdue.edu}

\author{Xu Wang}
\address{Department of Mathematics, Purdue University, West Lafayette, Indiana
47907, USA.}
\email{wang4191@purdue.edu}

%\thanks{The research is supported in part by the NSF grant DMS-1912704.}



\begin{abstract}
This paper is concerned with an inverse scattering problem for the time-harmonic
elastic wave equation with a random potential. Interpreted as a distribution, the
potential is assumed to be a microlocally isotropic generalized Gaussian random
field with the covariance operator being described by a classical
pseudo-differential operator. The goal is to determine the principal symbol of
the covariance operator from the scattered wave measured in a bounded domain
which has a positive distance from the domain of the potential. For such a rough
potential, the well-posedness of the direct scattering problem in the
distribution sense is established by studying an equivalent Lippmann--Schwinger
integral equation. For the inverse scattering problem, it is shown with
probability one that the principal symbol of the covariance operator can be
uniquely determined by the amplitude of the scattered waves averaged over the
frequency band from a single realization of the random potential. The analysis
employs the Born approximation in high frequency, asymptotics of the Green
tensor for the elastic wave equation, and microlocal analysis for the Fourier
integral operators. 
\end{abstract}

\maketitle

\section{Introduction and statement of the main result}

The scattering problems for elastic waves have significant applications in
diverse scientific areas such as geophysical exploration and nondestructive
testing \cite{ LL-86, UWW}. In medical diagnostics, elastography
is an emerging imaging modality that seeks to determine the mechanical
properties of elastic media from their response to exciting forces \cite{OCPY}.
By mapping the elastic properties and stiffness of soft tissues, it can give
diagnostic information about the presence or status of disease \cite{GFI}.
Driven by these applications, the underlying inverse problems, which are to
determine the medium properties based on the elastic wave equation, have been
extensively studied and many mathematical results are available, especially for
the uniqueness \cite{BFPRV, H1, H2, Ra-TAMS}. We refer to \cite{ABG-15} for a
comprehensive account of mathematical methods in elasticity imaging.  

Stochastic modeling has been widely adopted to handle problems involving
uncertainties and randomness. In the research area of wave propagation, the wave
fields may not be deterministic but rather are described by random fields due
to the uncertainties for the media and/or the environments. Therefore, it is
more appropriate to consider the stochastic wave equations to describe the wave
motion in random settings. In addition to the ill-posedness and nonlinearity, stochastic inverse
problems have substantially more difficulties than their deterministic
counterparts. The random parameters to be determined in stochastic inverse
problems can not be characterized by a particular realization, but instead, by
its statistics, such as expectation and covariance. Hence, the relationship
between these statistics and the wave fields needs to be established. In
general, the statistics of the data for the wave fields are required, which
significantly increases the computational cost since a large number of
realizations is needed. It is an important and challenging problem to determine
the statistics of the random parameters through fewer realizations of the wave
fields.

The paper is concerned with an inverse scattering problem for the time-harmonic
elastic wave equation with a random potential in two dimensions.
Specifically, we consider the stochastic elastic wave equation
\begin{eqnarray}\label{a1}
\mu\Delta \bm{u}+(\lambda+\mu)\nabla\nabla\cdot
\bm{u}+\omega^2\bm{u}-\rho\bm{u}=-\delta_y\bm{a} \quad {\rm in} ~ \mathbb R^{2},
\end{eqnarray}
where $\omega>0$ is the
angular frequency,  $\bm{a}$ is a unit polarization vector in
$\mathbb R^2$, $\delta_y(\cdot):=\delta(\cdot-y)$ is the Dirac delta function concentrated at the
source point $y\in\mathbb R^2$, $\lambda$ and $\mu$ are
the Lam\'{e} parameters satisfying $\mu>0$ and $\lambda+2\mu>0$ such that the
second-order partial differential operator
$\Delta^*:=\mu\Delta+(\lambda+\mu)\nabla\nabla\cdot$ is strongly elliptic
(cf. \cite[10.4]{M}).
The randomness of (\ref{a1}) comes from the potential $\rho$, which describes a
linear load inside a known homogeneous and isotropic elastic solid and is
considered to be a generalized Gaussian random field. Throughout, the
potential $\rho$ is required to satisfy the following assumption. 

\begin{assumption}\label{as:rho}
Assume that the centered random potential $\rho$ is a microlocally isotropic
Gaussian random field of order $-m$ in $D$ with $m\in(1,2]$ and $D$ being a bounded domain. More precisely, $\rho$ has the principal symbol $\phi(x)|\xi|^{-m}$, where $\phi$ is called the micro-correlation strength of the potential $
\rho$ and satisfies $\phi\in C_0^\infty(D), \phi\geq 0$.
\end{assumption}

The displacement of the total field $\bm u\in\mathbb C^2$ in (\ref{a1}) can be
decomposed into 
\begin{eqnarray*}%\label{a2}
\bm{u}(x,y)=\bm{u}^{ i}(x,y)+\bm{u}^{ s}(x,y),
\end{eqnarray*}
where $\bm{u}^{ s}$ represents the scattered field and $\bm{u}^{ i}$ is
the incident field given by
\begin{eqnarray*}%\label{a3}
\bm{u}^{ i}(x,y)=\bm{G}(x,y,\omega)\bm{a},\quad x\neq y.
\end{eqnarray*}
Here, $\bm{G}(x,y,\omega)\in\mathbb C^{2\times 2}$ denotes the Green tensor for the
Navier equation. Explicitly, 
\begin{eqnarray}\label{a4}
\bm{G}(x,y,\omega)=\frac{1}{\mu}\Phi(x,y,\kappa_{\rm
s})\bm{I}+\frac{1}{\omega^2}\nabla_x\nabla_x^\top\Big[\Phi(x,y,
\kappa_{\rm s})-\Phi(x,y,\kappa_{\rm p})\Big],
\end{eqnarray}
where $\bm{I}$ is the $2\times 2$ identity matrix, $\nabla_x=(\partial_{x_1},\partial_{x_2})^\top$ is the gradient operator,
$\Phi(x,y,\kappa)=\frac{\rm i}{4}H_0^{(1)}(\kappa|x-y|)$ is the fundamental
solution of the two-dimensional Helmholtz equation 
with $H_0^{(1)}$ being the Hankel function of the first kind with order zero, 
$\kappa_{\rm p}:=c_{\rm p}\omega$ and
$\kappa_{\rm s}:=c_{\rm s}\omega$ with $c_{\rm p}=(\lambda +2\mu)^{-\frac{1}{2}}$
and $c_{\rm s}=\mu^{-\frac{1}{2}}$ are known as the compressional and shear wavenumbers,
respectively.

Since the elastic wave equation (\ref{a1}) is imposed in the whole space
$\mathbb R^2$, an appropriate radiation condition is needed to complete the
problem formulation. By the Helmholtz decomposition (cf. \cite[Appendix
B]{BLZ}), the scattered field $\bm{u}^{ s}$ can be decomposed into 
the compressional wave component $\bm{u}^{ s}_{\rm p}$ and the shear wave
component $\bm{u}^{s}_{\rm s}$, i.e., 
\[
\bm{u}^{ s}=\bm{u}^{ s}_{\rm p}+\bm{u}^{ s}_{\rm s}
\quad\text{in} ~ \mathbb R^2\setminus \overline{D}.
\]
The Kupradze--Sommerfeld radiation condition requires that $\bm{u}^{ s}_{\rm
p}$ and $\bm{u}^{s}_{\rm s}$ satisfy the Sommerfeld radiation
condition
\begin{eqnarray}\label{a6}
\lim_{r\rightarrow\infty}r^{\frac{1}{2}}\left(\partial_r
\bm{u}^{ s}_{\rm p} -{\rm i}\kappa_{\rm p}\bm{u}^{ s}_{\rm
p}\right)=0,\quad \lim_{r\rightarrow\infty}r^{\frac{1}{2}}\left(\partial_r
\bm{u}^{s}_{\rm s} -{\rm i}\kappa_{\rm s}\bm{u}^{ s}_{\rm
s}\right)=0,\quad r=|x|.
\end{eqnarray}

As is known, the inverse scattering problems are challenging due to the nonlinearity and ill-posedness. Apparently, the stochastic inverse scattering problems are even harder in order to handle the extra difficulties of randomness and uncertainties. 
There are very few results concerning the solutions of the stochastic inverse scattering problems. For the inverse random source scattering problems, when the source is driven by an additive
white noise, effective mathematical models and efficient computational methods
have been proposed for the stochastic acoustic and elastic wave equations
\cite{BCL16, BCLZ, PL, LCL, BCL18, BX}. To determine the unknown parameters in
the above models, in general, the data of the expectation and variance for the
measured wave field is needed, and hence a fairly large number of realizations
of the random source is required. If the source is
described as a generalized Gaussian random field whose covariance is a
classical pseudo-differential operator, the results of uniqueness were
established in \cite{LHL, LL} for the stochastic acoustic and elastic wave
equations. It was shown that the principal symbol of the covariance operator can
be uniquely determined by the amplitude of the wave field averaged over the
frequency band. It is worth mentioning that the methods in \cite{LHL, LL} only
require the scattering data corresponding to a single realization of the random
source. For the random Schr\"{o}dinger equation in two dimensions where the
potential is a generalized Gaussian random field, it was proved in \cite{CHL}
and \cite{LPS08} that the principal symbol of the covariance operator can be
uniquely determined by the backscattered far-field data associated with the
plane wave and the scattered wave field associated with the point source,
respectively. Similarly, the approach only needs a single realization of the
random potential. A related work can be found in \cite{HLP}, where an inverse
scattering problem in a half-space with an impedance boundary condition was
studied where the impedance function is modeled as a generalized Gaussian random field.

In this work, we study both the direct and inverse scattering problems
for the stochastic elastic wave equation \eqref{a1} along with the radiation
condition \eqref{a6}. Given the random potential $\rho$, which may be so rough
that it can only be interpreted as a distribution, the direct scattering
problem is to determine the displacement $\bm{u}$ which satisfies \eqref{a1}
and \eqref{a6} in an appropriate sense. Using Green's theorem and the
Kupradze--Sommerfeld radiation condition, we deduce an equivalent
Lippmann--Schwinger integral equation. Based on the
Fredholm alternative theorem and the unique continuation principle,
the Lippmann--Schwinger equation is shown to have a unique solution in the
Sobolev space with a positive smoothness index. The inverse scattering problem
is to determine the micro-correlation strength $\phi(x)$ from the scattered field measured in a bounded and convex domain $U$
which has a positive distance from $D$, i.e., $U\subset\mathbb R^2\setminus
\overline{D}$.

It is clear to note from the elastic wave equation (\ref{a1}) that the
displacement $\bm u$ depends on the observation point $x$, the location of the
source point $y$, the angular frequency $\omega$, and the unit polarization
vector $\bm{a}$. To express explicitly the dependence of $\bm u$ on these
quantities, we write $\bm{u}(x,y,\omega,\bm{a})$,
$\bm{u}^{i}(x,y,\omega,\bm{a})$, $\bm{u}^{s}(x,y,\omega,\bm{a})$, and
$\bm{u}_j(x,y,\omega,\bm{a})$ in the Born series (cf. (\ref{c2}) and (\ref{c3})
for the definition of $\bm{u}_j$) for $\bm{u}(x,y)$, $\bm{u}^{i}(x,y)$,
$\bm{u}^{s}(x,y)$, and $\bm{u}_j(x,y)$, respectively. Moreover, when the
observation point $x$ coincides the source point $y$, for simplicity, we write
$\bm{u}^{s}(x,\omega,\bm{a})$ and $\bm{u}_j(x,\omega,\bm{a})$ for 
$\bm{u}^{s}(x,x,\omega,\bm{a})$ and $\bm{u}_j(x,x,\omega,\bm{a})$, respectively.

The following theorem concerns the uniqueness of the inverse scattering
problem and is the main result of this paper. 

\begin{theorem}\label{theorem1}
Let $\rho$ satisfy Assumption \ref{as:rho}
and additionally $m>\frac53$. Let $U\subset\mathbb R^2\backslash\overline{D}$ be
a bounded and convex domain having a locally Lipschitz boundary and a positive distance
from $D$. Then for all $x\in U$, it holds almost surely that 
\begin{eqnarray}\label{a7}
\lim_{Q\to\infty}\frac{1}{Q-1}\int_1^Q\omega^{m+2}\sum_{j=1}^2|\bm{u}^{
s}(x,\omega,\bm{a}_j)|^2d\omega=\frac{c_{\rm s}^{6-m}+c_{\rm p}^{6-m}}{2^{m+6}\pi^2}\int_{\mathbb
R^2}\frac{1}{|x-\zeta|^{2}}\phi(\zeta)d\zeta, 
\end{eqnarray}
where $\bm{a}_1$ and $\bm{a}_2$ are orthonormal vectors in $\mathbb R^2$. Moreover, the function
$\phi$ can be uniquely determined from the integral equation (\ref{a7}) for all
$x\in U$.
\end{theorem}

Since the scattered field $\bm{u}^{s}$ depends on the realization of the
random potential $\rho$, the scattering data given on the left hand side of
(\ref{a7}) is random for any finite $Q$. However, (\ref{a7}) indicates that the
scattering data is statistically stable when $Q$ approaches to infinity, i.e.,
it is independent of the realization of the potential. The main result
demonstrates that the function $\phi$ can be uniquely determined by the
amplitude of two scattered fields averaged over the frequency band,
which are generated by a single realization of the random potential. Here, the
two scattered fields are excited by the incident waves
$\bm{G}\bm{a}_1$ and $\bm{G}\bm{a}_2$. The proof of the main result is
quite technical. The analysis employs the Born approximation in the high frequency regime,
the asymptotics of Green's tensor for the elastic wave equation, and microlocal
analysis for the Fourier integral operators.

For readability, we briefly sketch the steps of the proof for the main result.
As mentioned above, the scattering problem (\ref{a1}) and (\ref{a6}) can be
equivalently formulated as a Lippmann--Schwinger integral equation which admits
a unique solution. A careful analysis shows that the Born series of the
Lippmann--Schwinger integral equation $\sum_{j=0}^{\infty}\bm{u}_j$ (cf. 
(\ref{c2}) and (\ref{c3}) for the definition of $\bm{u}_j$) converges to the
unique solution to the direct scattering problem when the angular frequency
$\omega$ is sufficiently large. Hence, the scattered field $\bm
u^s$ can be written as 
\begin{eqnarray*}
\bm{u}^{s}=\bm{u}_1+\bm u_2+\bm{b},\quad \bm{b}=\sum_{j=3}^{\infty}\bm{u}_j, 
\end{eqnarray*}

For the first item $\bm{u}_1$, by employing the asymptotic expansions of the
Green tensor and microlocal analysis for the Fourier integral operators via
multiple coordinate transformations, we show in Theorems \ref{thm3} that
\begin{eqnarray}\label{a9}
\lim_{Q\to\infty}\frac{1}{Q-1}\int_1^Q\omega^{m+2}\sum_{j=1}^2|\bm{u}_1
(x,\omega,\bm{a}_j)|^2d\omega=\frac{c_{\rm s}^{6-m}+c_{\rm p}^{6-m}}{2^{m+6}\pi^2}\int_{\mathbb
R^2}\frac{1}{|x-\zeta|^{2}}\phi(\zeta)d\zeta.
\end{eqnarray}
It is shown in Theorem \ref{thm4} that the contribution of the second item $\bm u_2$ can be neglected, i.e., 
\begin{eqnarray}
\lim_{Q\to\infty}\frac{1}{Q-1}\int_1^Q\omega^{m+2}|\bm{u}_2(x,\omega,
\bm{a})|^2d\omega=0.  
\end{eqnarray}
For the remaining term $\bm{b}$, by means of estimating the order with respect to
the angular frequency $\omega$, we deduce in Section \ref{sec:5.3} that 
\begin{eqnarray}\label{a11}
\lim_{Q\to\infty}\frac{1}{Q-1}\int_1^Q\omega^{m+2}|\bm{b}(x,\omega,
\bm{a})|^2d\omega=0.  
\end{eqnarray}
Finally, the main result follows from (\ref{a9})--(\ref{a11}) and the
Cauchy--Schwartz inequality.

The paper is organized as follows. In Section \ref{sec:2}, we briefly introduce the
microlocally isotropic generalized Gaussian random fields and present some of
their properties. Section \ref{sec:3} concerns the well-posedness of the direct
scattering problem. We show that the direct problem is equivalent to a
Lippmann--Schwinger integral equation which is uniquely solvable for a
distributional potential. In Section \ref{sec:4}, the Born series is studied for the
Lippmann--Schwinger integral equation in the high frequency regime. Sections \ref{sec:5}
is devoted to the inverse scattering problems. The paper is concluded with some general remarks and
directions for future work in Section \ref{sec:6}.

\section{Generalized Gaussian random fields}
\label{sec:2}

In this section, we give a brief introduction to the microlocally
isotropic generalized Gaussian random fields in $\mathbb R^d$, $d=2$ or $3$. Let $C_0^{\infty}(\mathbb R^d)$
be the set of smooth functions with compact support, and
$\mathcal{D}:=\mathcal{D}(\mathbb{R}^d)$ be the space of test functions, which
is $C_0^{\infty}(\mathbb R^d)$ equipped with a locally convex topology.
The dual space $\mathcal{D'}:=\mathcal{D'}(\mathbb R^d)$ of $\mathcal{D}$ is
called the space of distributions on $\mathbb{R}^d$ and is equipped with a
weak-star topology (cf. \cite{AF03}). Denote by $(\Omega,
\mathcal{F},\mathbb{P})$ a complete probability space, where $\Omega$ is a samples pace, $\mathcal F$ is a $\sigma$-algebra on $\Omega$, and $\mathbb P$ is a probability measure on the measurable space $(\Omega, \mathbb F)$. 

A function $\rho$ is said to be a generalized random field if, for each
$\tilde{\omega}\in\Omega$, the realization $\rho(\tilde{\omega})$ belongs to
$\mathcal{D'}(\mathbb R^d)$ and the mapping
\begin{eqnarray}\label{l1}
\tilde{\omega}\in\Omega\longmapsto \langle \rho(\tilde{\omega}),\psi\rangle\in\mathbb R
\end{eqnarray}
is a random variable for all $\psi\in \mathcal{D}$, where
$\langle\cdot,\cdot\rangle$ denotes the dual product between $\mathcal{D}'$ and
$\mathcal{D}$. The distributional derivative of $\rho\in\mathcal{D}'$ is defined
by
\[
\langle\partial_{x_j}\rho,\psi\rangle=-\langle
\rho,\partial_{x_j}\psi\rangle\quad\forall\,\psi\in\mathcal{D},\quad j=1,
\dots, d. 
\]
A generalized random field is said to be Gaussian if (\ref{l1}) defines
a Gaussian random variable for all $\psi\in \mathcal{D}$.

For a generalized random field $\rho\in\mathcal{D}'$, we can define its
expectation $\mathbb E\rho\in\mathcal{D}'$ and covariance operator $Q_{\rho}:
\mathcal{D}\rightarrow \mathcal{D'}$ as follows:
\begin{align*}
\langle\mathbb{E}\rho,\psi\rangle &:=\mathbb{E}\langle \rho,\psi\rangle\quad\forall\,\psi\in\mathcal{D},\\
\langle Q_{\rho}\psi_1,\psi_2\rangle &:={\rm Cov}(\langle \rho,\psi_1\rangle,\langle \rho,\psi_2\rangle)=\mathbb E\left[(\langle
\rho,\psi_1\rangle-\mathbb E\langle \rho,\psi_1\rangle)(\langle
\rho,\psi_2\rangle-\mathbb E\langle \rho,\psi_2\rangle)\right]\quad\forall\,\psi_1,\psi_2\in\mathcal{D}.
\end{align*}
It follows from the continuity of $Q_\rho$ and the Schwartz kernel theorem that there exists a unique kernel function $\mathcal{K}_\rho(x,y)$ satisfying 
\begin{eqnarray*}%\label{l3}
\langle  {Q}_{\rho}\psi_1,\psi_2\rangle=\int_{\mathbb{R}^d}\int_{\mathbb{R}^d}\mathcal{K}_\rho(x,y)\psi_1(x)\psi_2(y)dxdy\quad\forall\,\psi_1,\psi_2\in \mathcal{D}.
\end{eqnarray*} 
The following definition can be found in \cite{LPS08} on the microlocally
isotropic generalized Gaussian random fields. 

\begin{definition}\label{def1}
A generalized Gaussian random field $\rho$ on $\mathbb R^d$ is called
microlocally isotropic of order $-m$ in $D$ with $m\ge0$ if the
realizations of $\rho$ are almost surely supported in $D$ and its covariance
operator $Q_\rho$ is a classical pseudo-differential operator having an
isotropic principal symbol $\phi(x)|\xi|^{-m}$ with $\phi\in
C_0^\infty(\mathbb{R}^d)$, $supp\,\phi\subset D$ and $\phi\ge0$.
\end{definition}

Without loss of generality, we choose the bounded domain $D$ which not only
contains the support of $\rho$ almost surely but also has a locally Lipschitz
boundary.

To have a better understanding of microlocally isotropic Gaussian random fields,
we give an example by introducing the centered fractional Gaussian fields (cf.
\cite{LW, LSSW16}) defined by
\begin{align}\label{eq:FGF}
f_m(x):=(-\Delta)^{-\frac m4}\dot{W}(x),\quad x\in\mathbb{R}^d,
\end{align}
where $(-\Delta)^{-\frac m4}$ is the fractional Laplacian and
$\dot{W}\in\mathcal{D}'$ denotes the white noise. It is shown in \cite{LW} that
the kernel $\mathcal{K}_{f_m}(x,y)$ of $f_m$ is isotropic since its value
depends only on the distance between $x$ and $y$, and $f_m$ is a
microlocally isotropic Gaussian random field of order $-m$ and
satisfies Definition \ref{def1} with $\phi\equiv1$.
In particular, if $m\in(d,d+2)$, the fractional Gaussian field $f_m$ defined above 
is a translation of the classical fractional Brownian motion. More precisely, 
\[
\tilde f_m(x):=\langle f_m,\delta_x-\delta_0\rangle,\quad x\in\mathbb{R}^d
\]
has the same distribution as the classical fractional Brownian motion with Hurst
parameter $H=\frac{m-d}2\in(0,1)$ up to a multiplicative constant.

The regularity and kernel functions can be obtained for the microlocally
isotropic Gaussian random fields by using the relationship between
them and the fractional Gaussian fields defined in \eqref{eq:FGF}. It is clear
to note that the fractional Gaussian field $f_m$ defined by \eqref{eq:FGF} has
the same regularity as the microlocally isotropic Gaussian random field $\rho$
of order $-m$ in Definition \ref{def1}. Hence, we may have the following
regularity results for the microlocally isotropic Gaussian random fields. The
proof can be found in \cite{LW}.

\begin{lemma}\label{le1}
Let $\rho$ be a microlocally isotropic Gaussian random field of
order $-m$ in $D$ with $m\in[0,d+2)$. 
\begin{itemize}
\item[(i)] If $m\in(d,d+2)$, then $\rho\in C^{0,\alpha}(D)$ almost surely for
all $\alpha\in(0,\frac{m-d}2)$.

\item[(ii)] If $m\in[0,d]$, then $\rho\in W^{\frac{m-d}2-\epsilon,p}(D)$ almost
surely for any $\epsilon>0$ and $p\in(1,\infty)$.
\end{itemize}
\end{lemma}

Moreover, the kernels for both $\rho$ in Definition \ref{def1} and $f_m$ defined
in \eqref{eq:FGF} are isotropic and have the same order. The following result
gives the explicit expressions of the kernels for $f_m$. The proof can be
found in \cite{LSSW16}. 

\begin{lemma}\label{le2}
Let $f_m$ be a fractional Gaussian field defined by \eqref{eq:FGF}. Denote
$H=\frac{m-d}2$. The kernel function $\mathcal{K}_{f_m}$ of $f_m$ has the
following form:

\begin{itemize}
\item[(i)] If $m\in(0,\infty)$ and $H$ is not a nonnegative integer, then
\[
\mathcal{K}_{f_m}(x,y)=C_1(m,d)|x-y|^{2H},
\]
where  $C_1(m,d)=2^{-m}\pi^{-\frac
d2}\Gamma(\frac{d-m}2)/\Gamma(\frac m2)$ with $\Gamma(\cdot)$ being the Gamma
function.

\item[(ii)] If $m\in(0,\infty)$ and $H$ is a nonnegative integer, then
\[
\mathcal{K}_{f_m}(x,y)=C_2(m,d)|x-y|^{2H}\ln|x-y|,
\]
where 
$C_2(m,d)=(-1)^{H+1}2^{-m+1}\pi^{-\frac d2}/(H!\Gamma(\frac m2)).$

\item[(iii)] If $m=0$, then 
\[
\mathcal{K}_{f_m}(x,y)=\delta(x-y),
\]
where $\delta(\cdot)$ is the Dirac delta function centered at $0$.
\end{itemize}
\end{lemma}

We conclude this section by giving the kernel function of a microlocally
isotropic Gaussian random field in Definition \ref{def1}, which has the
form
\[
\mathcal{K}_\rho(x,y)=\phi(x)\mathcal{K}_{f_m}(x,y)+h(x,y),
\]
where $\phi\mathcal{K}_{f_m}$ is the leading term with strength
$\phi$ and $h$ is a smooth residual (cf. \cite{LPS08}).

\section{The direct scattering problem}
\label{sec:3}
According to Lemma \ref{le1} with $d=2$, if $m\in(2,4)$, the random potential 
$\rho$ is a H\"{o}lder continuous function almost surely and has enough
regularity such that the scattering problem \eqref{a1} and \eqref{a6} is
well-posed in the traditional sense (cf. \cite{BLZ}). However, if $m\in[0,2]$,
then the random potential $\rho\in W^{\frac{m-2}2-\epsilon,p}(D)$ is
a distribution, and the elastic wave equation \eqref{a1} should be considered in
the distribution sense instead. 

In this section, we study the well-posedness of the scattering problem \eqref{a1} and \eqref{a6} 
with $m\in[0,2]$ (cf. Assumption \ref{as:rho}) by considering the equivalent Lippmann--Schwinger
integral equation. 

In the sequel, we denote by $\bm{X}:=X^2=\{\bm{g}=(g_1,g_2)^\top:g_j\in
X,~\forall\, j=1,2\}$ the Cartesian product vector space of $X$, and use
the notation $\bm{W}^{r,p}:=(W^{r,p}(\mathbb{R}^2))^2$ and
$\bm{H}^r:=\bm{W}^{r,2}$ for simplicity. The notation $a\lesssim b$ or $a
\gtrsim b$ stands for $a\leq C b$ or $a\geq C b$, where $C$ is a positive
constant whose value is not required but
should be clear from the context. 

\subsection{The Lippmann--Schwinger integral equation}

Based on the Green tensor $\bm G$ given in \eqref{a4} and given a source point
$y\in\mathbb R^2$, the Lippmann--Schwinger integral equation takes the form
\begin{eqnarray}\label{eq:LS}
\bm{u}(x,y)+\int_D\bm{G}(x,z,\omega)\rho(z)\bm{u}(z,y)dz=\bm{G}(x,y,\omega)\bm{a},\quad
x\in\mathbb R^2, \quad x\neq y. 
\end{eqnarray}
For a fixed $y\in\mathbb{R}^2$, define two scattering operators $H_\omega$ and
$K_\omega$ by 
\[
(H_\omega \bm u)(x):=[H_\omega\bm u(\cdot,y)](x)=\int_{\mathbb{R}^2}\bm{G}(x,z,\omega)\bm{u}(z,y)dz
\]
and
\begin{align}\label{eq:K}
(K_\omega \bm{u})(x):=[K_\omega\bm u(\cdot,y)](x)=\int_{\mathbb{R}^2}\bm{G}(x,z,\omega)\rho(z)\bm{u}(z,y)dz,
\end{align}
which have the following properties (cf. \cite[Lemma 4.2]{LW2}).

\begin{lemma}\label{lm:operators}
Let $\rho$ satisfy Assumption \ref{as:rho},  $\mathcal{O}\subset\mathbb{R}^2$ be a bounded set, and $\mathcal{V}\subset\mathbb{R}^2$ be a bounded open set with a locally Lipschitz boundary.
\begin{itemize}
\item[(i)] The operator
$H_\omega:\bm{H}_0^{-\beta}(\mathcal{O})\to\bm{H}^\beta(\mathcal{V})$ is bounded
for any $\beta\in(0,1]$. 

\item[(ii)] The operator
$H_\omega:\bm{W}_0^{-\gamma,p}(\mathcal{O})\to\bm{W}^{\gamma,q}(\mathcal{V})$ is compact for any $q\in(2,\infty)$, $\gamma\in\big(0,\frac2q\big)$ and $p$ satisfying $\frac1p+\frac1q=1$.

\item[(iii)] The operator $K_\omega:\bm{W}^{\gamma,q}(\mathcal{V})\to\bm{W}^{\gamma,q}(\mathcal{V})$ is compact for any $q\in\big(2,\frac{4}{2-m}\big)$ and $\gamma\in\big(\frac{2-m}2,\frac2q\big)$.

\end{itemize}
\end{lemma}

The following result gives the well-posedness of the Lippmann--Schwinger
integral equation \eqref{eq:LS}. 

\begin{theorem}\label{tm:LS}
Let $\rho$ satisfy Assumption \ref{as:rho}. Then the Lippmann--Schwinger integral equation
\eqref{eq:LS} admits a unique solution $\bm{u}\in\bm{W}^{\gamma,q}_{loc}$ almost surely with $q\in\big(2,\frac{2}{2-m}\big)$ and
$\gamma\in\big(\frac{2-m}2,\frac1q\big)$.
\end{theorem}

\begin{proof}
Let $\mathcal{V}\subset\mathbb{R}^2$ be any bounded open set with a locally
Lipschitz boundary. By the definition of the operator $K_\omega$, the 
Lippmann--Schwinger integral equation \eqref{eq:LS} can be written as
\begin{align}\label{eq:LSoperators}
[(I+K_\omega)\bm{u}(\cdot,y)](x)=\bm{G}(x,y,\omega)\bm{a},\quad x\in\mathbb{R}^2,
\end{align}
where $y\in\mathbb{R}^2$ is fixed and $I$ is the identity operator. It follows
from  Lemma \ref{lm:operators} that the operator
$I+K_\omega:\bm{W}^{\gamma,q}(\mathcal{V})\to\bm{W}^{\gamma,q}(\mathcal{ V})$ is
Fredholm. Moreover, it is shown in \cite[Lemma 4.1]{LHL} that
$\bm{G}(\cdot,y,\omega)\in (W^{1,p'}(\mathcal{V}))^{2\times 2}$ with
$p'\in(1,2)$. Choosing $p'=2-\epsilon$ with $\epsilon>0$
being sufficiently small, we obtain from  the Kondrachov compact embedding
theorem that the embedding
$\bm{W}^{1,p'}(\mathcal{V})\hookrightarrow\bm{W}^{\gamma,q}(\mathcal{V})$ is
compact, which indicates that the right-hand side of \eqref{eq:LSoperators}
satisfies $\bm{G}(\cdot,y,\omega)\bm{a}\in\bm{W}^{\gamma,q}(\mathcal{V}).$

By the Fredholm alternative theorem, the 
Lippmann--Schwinger integral equation \eqref{eq:LSoperators} has a
unique solution $\bm{u}\in\bm{W}^{\gamma,q}(\mathcal{V})$ if 
\begin{align}\label{eq:LShomo}
(I+K_\omega)\bm{u}=0
\end{align}
has only the trivial solution $\bm{u}\equiv0$. This fact can be proved by the unique continuation principle (cf. \cite{LW2}), which restrict the parameters $q$ and $\gamma$ to intervals $\big(2,\frac{2}{2-m}\big)$ and
$\big(\frac{2-m}2,\frac1q\big)$, respectively.
\end{proof}

\subsection{Well-posedness}

Now we show the existence and uniqueness of the solution of \eqref{a1} in the
distribution sense by utilizing the Lippmann--Schwinger integral equation.

\begin{theorem}\label{tm:wellposed}
Let $\rho$ satisfy Assumption \ref{as:rho}. The elastic
wave equation \eqref{a1} together with the radiation condition \eqref{a6} is
well-defined in the distribution sense, and admits a unique solution
$\bm{u}\in\bm{W}^{\gamma,q}_{loc}$ almost surely with $q\in\big(2,\frac{2}{2-m}\big)$ and
$\gamma\in\big(\frac{2-m}2,\frac1q\big)$.
\end{theorem}

\begin{proof}
First we show the existence of the solution of \eqref{a1}. It suffices to 
show that the solution of the Lippmann--Schwinger integral equation
\eqref{eq:LS} is also a solution of \eqref{a1} in the distribution sense.

Suppose that $\bm{u}^*\in\boldsymbol{W}^{\gamma,q}_{loc}$ is the solution of
\eqref{eq:LS} and satisfies 
\[
\bm{u}^*(x,y)+\int_{\mathbb{R}^2}\bm{G}(x,z,\omega)\rho(z)\bm{u}^*(z,y)dz=\bm{G}(x,y,\omega)\bm{a},\quad x\in\mathbb{R}^2.
\]
Since the Green tensor $\bm{G}$ is the fundamental solution for the operator
$\Delta^*+\omega^2 $, we have 
\[
(\Delta^*+\omega^2)\bm{G}(\cdot,y,\omega)=-\delta_y\bm{I},
\]
where $\bm{I}$ is the $2\times 2$ identity matrix, $\delta_y$ is a distribution, i.e.,
$\delta_y\in\mathcal{D}'$. Hence, we get for any
$\boldsymbol{\psi}\in\boldsymbol{\mathcal{D}}$ that  
\[
\langle(\Delta^*+\omega^2)\bm{G}(\cdot,y,\omega),\boldsymbol{\psi}\rangle=-\langle\delta_y\bm{I},\boldsymbol{\psi}\rangle=-\boldsymbol{\psi}(y).
\]
For any $\boldsymbol{\psi}\in\boldsymbol{\mathcal{D}}$, a simple calculation
yields 
\begin{align*}
&\langle\mu\Delta\bm{u}^*+(\lambda+\mu)\nabla\nabla\cdot\bm{u}^*+\omega^2\bm{u}^*-\rho\bm{u}^*,\boldsymbol{\psi}\rangle\\
=&-\Big\langle\int_{\mathbb{R}^2}\left(\Delta^*+\omega^2\right)\bm{G}(\cdot,z
,\omega)\rho(z)\bm{u}^*(z,y)dz,\boldsymbol{\psi}\Big\rangle\\
&+\left\langle\left(\Delta^*+\omega^2\right)\bm{G}(\cdot,y,\omega)\bm{a},\boldsymbol{\psi}\right\rangle
-\langle \rho\bm{u}^*,\boldsymbol{\psi}\rangle\\
=&-\int_{\mathbb{R}^2}\left(\rho(z)\bm{u}^*(z,y)\right)^\top\left\langle\left(\Delta^*+\omega^2\right)\bm{G}(\cdot,z,\omega),\boldsymbol{\psi}\right\rangle dz\\
&+\bm{a}^\top\left\langle\left(\Delta^*+\omega^2\right)\bm{G}(\cdot,y,\omega),\boldsymbol{\psi}\right\rangle -\langle \rho\bm{u}^*,\boldsymbol{\psi}\rangle\\
=&\int_{\mathbb{R}^2}\left(\rho(z)\bm{u}^*(z,y)\right)^\top\boldsymbol{\psi}(z)dz-\bm{a}^\top\boldsymbol{\psi}(y)-\langle \rho\bm{u}^*,\boldsymbol{\psi}\rangle\\
=&-\langle \delta_y\bm{a},\boldsymbol{\psi}\rangle, 
\end{align*}
which implies that $\bm{u}^*\in\bm{W}^{\gamma,q}_{loc}$ is also a solution of
\eqref{a1} and shows the existence of the solution of \eqref{a1}
according to Theorem \ref{tm:LS}.

The uniqueness of the solution of \eqref{a1} is obtained by using the same
procedure as that of the Lippmann--Schwinger equation. It requires to show
that if $\bm{a}=0$, then any solution $\bm{u}$ of the homogeneous equation
\eqref{a1} in the distribution sense is also a solution of \eqref{eq:LShomo}
with $\bm{a}=0$, i.e., $\bm{u}\equiv0$. 
In fact, let $\bm{u}$ be any solution of \eqref{a1} with $\bm{a}=0$. Then $\bm{u}$ satisfies
\[
\Delta^*\bm{u}+\omega^2\bm{u}=\rho\bm{u}
\]
in the distribution sense, where $\rho\in W^{\frac{m-d}2-\epsilon,p'}_0(D)\hookrightarrow W_0^{-\gamma,\tilde
p}(D)$ for some $p'>1$ and $\tilde p=\frac p{2-p}$ with $p$ satisfying $\frac1p+\frac1q=1$, $\bm{u}\in\bm{W}^{\gamma,q}_{loc}$ and
$\rho\bm{u}\in\bm{W}_0^{-\gamma,p}(D)$ according to the proof of Lemma
\ref{lm:operators}.
Let $B_r$ be an open ball with radius $r$ large enough such that $D\subset
B_r$. It follows from the proof of Theorem \ref{tm:LS} that
$\bm{G}(\cdot,y,\omega)\in\left(W^{\gamma,q}(B_r)\right)^{2\times 2}$. Hence, we get
\begin{eqnarray}\label{b12}
\int_{B_r}\bm{G}(x,z,\omega)\left[\Delta^*\bm{u}(z)+\omega^2\bm{u}(z)\right]dz
=\int_{B_r}\bm{G}(x,z,\omega)\rho(z)\bm{u}(z)dz.
\end{eqnarray}

Denote by $T$ the operator that maps $\bm{u}$ to the left-hand side of
(\ref{b12}). For $\bm{\psi}\in \bm{\mathcal{D}}$, by the similar arguments
as those in the proof of \cite[Lemma 4.3]{LHL}, we obtain
\begin{eqnarray*}%\label{b13}
(T\bm{\psi})(x)=-\bm{\psi}(x)+\int_{\partial
B_r}\left[\bm{G}(x,z,\omega)P\bm{\psi}(z)-P\bm{G}(x,z,\omega)\bm{\psi}
(z)\right]ds(z),
\end{eqnarray*}
where $P$ is the generalized stress vector on $\partial B_r$ defined by
$P\bm{\psi}:=\mu\frac{\partial\bm{\psi}}{\partial\nu}+(\lambda+\mu)(\nabla\cdot\bm{\psi})\nu$
with $\nu$ being the unit outward normal vector on the boundary $\partial B_r$.
Since $\bm{u}$ can be approximated by smooth functions, we have 
\begin{eqnarray*}%\label{b14}
-\bm{u}(x)+\int_{\partial B_r}\Big[\bm{G}(x,z,\omega)P\bm{u}(z)-P\bm{G}(x,z,\omega)\bm{u}(z)\Big]ds(z)=\int_{B_r}\bm{G}(x,z,\omega)\rho(z)\bm{u}(z)dz.
\end{eqnarray*}
Letting $r\to\infty$ and using the radiation condition, we get
\begin{eqnarray*}
\bm{u}(x)=-\int_{\mathbb R^2}\bm{G}(x,z,\omega)\rho(z)\bm{u}(z)dz,
\end{eqnarray*}
which indicates that $\bm{u}$ is also a solution of the Lippmann--Schwinger
equation \eqref{eq:LS} with $\bm{a}=0$, and hence $\bm{u}\equiv0$ according to
Theorem \ref{tm:LS}.
\end{proof}

\section{The Born series} 
\label{sec:4}

The results in the previous section indicate that the scattering
problem \eqref{a1} and \eqref{a6}, which is interpreted in the distribution
sense, is equivalent to the Lippmann--Schwinger integral equation \eqref{eq:LS}.
In the sequel, we may just focus on the Lippmann--Schwinger integral equation
\eqref{eq:LS}.

To get an explicit expression of the solution, we consider the Born sequence of
the Lippmann--Schwinger integral equation
\begin{eqnarray}\label{c2}
\bm{u}_j(x,y)=[-K_\omega\bm{u}_{j-1}(\cdot,y)](x),\quad{ j\in\mathbb{N}},
\end{eqnarray}
where the leading term is 
\begin{eqnarray}\label{c3}
\bm{u}_0(x,y)=\bm{G}(x,y,\omega)\bm{a}.
\end{eqnarray}

The goal of this section is to prove that the Born
series $\sum_{j=0}^{\infty}\bm{u}_j$ converges to the solution $\bm{u}$ for
sufficiently large $\omega$. 

\subsection{Estimates of the scattering operators}

Before showing the convergence of the Born series, we first introduce a 
weighted space which is equipped with a weighted $L^p$-norm (cf. \cite{LLM}).
For any $\delta\in\mathbb{R}$, let
\[
L_\delta^p(\mathbb{R}^2):=\{f\in L_{loc}^1(\mathbb{R}^2): \|f\|_{L_\delta^p}<\infty\},
\]
which is denoted by $L_\delta^p$ for short and is equipped with the norm
\[
\|f\|_{L_\delta^p}:=\|(1+|\cdot|^2)^{\frac{\delta}2}f\|_{L^p}
=\Big(\int_{\mathbb{R}^2}(1+|x|^2)^{\frac{\delta
p}2}|f(x)|^pdx\Big)^{\frac1p}.
\]
Define the space
\[
H^{s,p}_{\delta}(\mathbb{R}^2):=\{f\in\mathcal{S}':(I-\Delta)^{\frac s2}f\in L_\delta^p\},
\]
which is denoted by $H_\delta^{s,p}$ for short if there is no ambiguity and
is equipped with the norm 
\[
\|f\|_{H^{s,p}_\delta}=\|(I-\Delta)^{\frac s2}f\|_{L_\delta^p}.
\]
Here $\mathcal{S}'$ denotes the dual space of $\mathcal{S}$ which is the space
of all rapidly decreasing functions. When $\delta=0$, the space $H_0^{s,p}$ can be identified with the classical Sobolev space $W^{s,p}$.
When $p=2$, for simplicity, denote 
$H^s_\delta:=H^{s,2}_\delta$. For any $s\in\mathbb{R}$ and $\delta\in[0,1]$, it
is easy to verify that 
\begin{align}\label{eq:Hsdelta}
\|f\|_{H^s_\delta}&=\|(I-\Delta)^{\frac
s2}f\|_{L^2_\delta}=\|(1+|\cdot|^2)^{\frac{\delta}2}(I-\Delta)^{\frac
s2}f\|_{L^2}\nonumber\\
&=\|(I-\Delta)^{\frac{\delta}2}(1+|\cdot|^2)^{\frac s2}\hat
f\|_{L^2}\\
&=\|(1+|\cdot|^2)^{\frac s2}\hat f\|_{H^\delta}\nonumber\\
&\gtrsim\|(1+|\cdot|^2)^{\frac s2}\hat
f\|_{L^2}=\|f\|_{H^s}\nonumber,
\end{align}
where we have used \cite[Theorem 13.5]{E11} to obtain the inequality.

Based on these weighted norms, the operators $H_\omega$ and $K_\omega$ can
be estimated as follows.

\begin{lemma}\label{lm:H}
Let $\mathcal V\subset\mathbb R^2$ be any bounded domain. For any $s\in(0,\frac12)$ and $\epsilon>0$, the following estimates hold:
\begin{align}\label{y1}
\|H_{\omega}\|_{\mathcal{L}(\bm H_1^{-s}, \bm H_{-1}^{s})} &\lesssim  
\omega^{-1+2s},\\\label{y2}
\|H_{\omega}\|_{\mathcal{L}(\bm H_1^{-s}, \bm L^{\infty}(\mathcal V))} &\lesssim 
\omega^{s+\epsilon}.
\end{align}
\end{lemma}

\begin{proof}
The Green tensor ${\bm G}(x,y):=\bm G(x,y,\omega)$ satisfies 
\begin{eqnarray}\label{x31}
\mu\Delta \bm{G}(x,y)+(\lambda+\mu)\nabla\nabla\cdot
\bm{G}(x,y)+\omega^2\bm{G}(x,y)=-\delta(x-y)\bm{I} \quad {\rm in} ~ \mathbb R^{2}.
\end{eqnarray}
Taking the Fourier transform on both sides of \eqref{x31} with respect to $x-y$
leads to 
\begin{eqnarray*}
-\mu|\xi|^2\widehat{\bm G}(\xi)-(\lambda+\mu)\xi\cdot\xi^{\top}\widehat{\bm
G}(\xi)+\omega^2\widehat{\bm G}(\xi)=-\bm{I},
\end{eqnarray*}
where $\xi=(\xi_1, \xi_2)^\top$. A simple calculation gives 
\begin{equation}\label{x1}
 \widehat{{\bm G}}(\xi)=\frac{c_{\rm s}^2c_{\rm p}^2}{(|\xi|^2-c_{\rm
s}^2\omega^2)(|\xi|^2-c_{\rm p}^2\omega^2)} A(\xi),
\end{equation}
where the matrix 
\[
A(\xi):=\begin{bmatrix} \mu
|\xi|^2-\omega^2+(\lambda+\mu)\xi_2^2 &
-(\lambda+\mu)\xi_1\xi_2\\[3pt] -(\lambda+\mu)\xi_1\xi_2 &
\mu |\xi|^2-\omega^2+(\lambda+\mu)\xi_1^2
\end{bmatrix}.
\]

Let ${\bm f}=(f_1, f_2)^\top\in \bm{C}_0^{\infty}$ and ${\bm g}=(g_1,
g_2)^\top\in \bm{C}_0^\infty$. We have from the Parseval identity that 
\begin{align}
\langle H_{\omega}{\bm f}, {\bm g}\rangle&=\int_{\mathbb R^2}
\widehat{H_{\omega}{\bm f}}(\xi)\hat{{\bm g}}(\xi)d\xi=\int_{\mathbb
R^2}\widehat{{\bm G}}(\xi)\hat{{\bm f}}(\xi)\hat{{\bm
g}}(\xi)d\xi\nonumber\\
&=\int_{\mathbb R^2}\Big(\big[\widehat{
G_{11}}(\xi)\hat{f}_1(\xi)+\widehat{G_{12}}(\xi)\hat{f}_2(\xi)\big]\hat{g}_1(\xi)\nonumber\\\label{x2}
&\quad+\big[\widehat{G_{21}}(\xi)\hat{f}_1(\xi)+\widehat{G_{22}}(\xi)\hat{f}_2(\xi)\big]\hat{g}_2(\xi)\Big)d\xi,
\end{align}
where $\widehat{G_{ij}}$ denotes the $(i,j)$-entry of $\widehat{\bm
G}$. Noting that each term in \eqref{x2} has the same singularity at the
points $|\xi|=c_{\rm s}\omega$ and $|\xi|=c_{\rm p}\omega$, we only need to 
estimate the terms
\begin{align}\label{x3}
\int_{\mathbb R^2}\widehat{G_{11}}(\xi)\hat{f}_1(\xi)\hat{g}_1(\xi)d\xi&= c_{\rm s}^2c_{\rm
p}^2\int_{\mathbb R^2}\frac{\mu
|\xi|^2-\omega^2+(\lambda+\mu)\xi_2^2}{(|\xi|^2-c_{\rm
s}^2\omega^2)(|\xi|^2-c_{\rm
p}^2\omega^2)}\hat{f}_1(\xi)\hat{g}_1(\xi)d\xi,\\\label{x4}
\int_{\mathbb R^2}\widehat{G_{12}}(\xi)\hat{f}_2(\xi)\hat{g}_1(\xi)d\xi&= c_{\rm s}^2c_{\rm
p}^2\int_{\mathbb R^2}\frac{-(\lambda+\mu)\xi_1\xi_2}{(|\xi|^2-c_{\rm
s}^2\omega^2)(|\xi|^2-c_{\rm
p}^2\omega^2)}\hat{f}_2(\xi)\hat{g}_1(\xi)d\xi,
\end{align}
and the other two terms can be estimated similarly.

Define the Bessel potential operator $\mathcal J^s$ by
\begin{eqnarray*}%\label{x5}
\mathcal J^s h(x)=\mathcal{F}^{-1}[(1+|\xi|^2)^{\frac{s}{2}}\hat{h}(\xi)]
(x)\quad\forall\, s\in{\mathbb R},\, h\in\mathcal{S},
\end{eqnarray*}
where $\mathcal F^{-1}$ is the inverse Fourier transform. To deal with the
singularity, we split the whole space $\mathbb R^2$ into three parts: 
\begin{eqnarray*}
&&\Omega_1:=\{\xi\in {\mathbb R^2}: \|\xi|-c_{\rm
s}\omega|<\varepsilon_1\omega\},\\
&&\Omega_2:=\{\xi\in {\mathbb R^2}: \|\xi|-c_{\rm
s}\omega|>\varepsilon_1\omega\;\;{\rm and}\;\;\|\xi|-c_{\rm
p}\omega|<\varepsilon_2\omega\},\\
&&\Omega_3:=\{\xi\in {\mathbb R^2}: \|\xi|-c_{\rm
s}\omega|>\varepsilon_1\omega\;\;{\rm and}\;\;\|\xi|-c_{\rm
p}\omega|>\varepsilon_2\omega\},
\end{eqnarray*}
where $\varepsilon_1>0$ and $\varepsilon_2>0$ are two constants. 

First we estimate \eqref{x3}. Let
\begin{align*}
{\rm I}_j &=c_{\rm s}^2c_{\rm p}^2\int_{\Omega_j}\frac{\mu
|\xi|^2-\omega^2+(\lambda+\mu)\xi_2^2}{(|\xi|^2-c_{\rm
s}^2\omega^2)(|\xi|^2-c_{\rm
p}^2\omega^2)}\hat{f}_1(\xi)\hat{g}_1(\xi)d\xi\\
&=c_{\rm s}^2c_{\rm p}^2\int_{\Omega_j}\frac{\mu
|\xi|^2-\omega^2+(\lambda+\mu)\xi_2^2}{(|\xi|^2-c_{\rm
s}^2\omega^2)(|\xi|^2-c_{\rm
p}^2\omega^2)}(1+|\xi|^2)^s\widehat{\mathcal
J^{-s}f_1}(\xi)\widehat{\mathcal J^{-s}g_1} (\xi)d\xi ,\quad j=1,2,3.
\end{align*}

For the term ${\rm I}_3$, using the definition of $\Omega_3$  and noting 
\begin{eqnarray*}
|(\lambda+\mu)\xi_2^2|=|(\lambda+2\mu)\xi_2^2-\mu\xi_2^2|\leq (\lambda+2\mu)|\xi|^2-\omega^2+\mu|\xi|^2-\omega^2+2\omega^2
\end{eqnarray*}
and
\begin{eqnarray*}
\left|\frac{\mu |\xi|^2-\omega^2+(\lambda+\mu)\xi_2^2}{(|\xi|^2-c_{\rm s}^2\omega^2)(|\xi|^2-c_{\rm p}^2\omega^2)}\right|\lesssim \frac{1}{\|\xi|^2-c_{\rm p}^2\omega^2|}+ \frac{1}{\|\xi|^2-c_{\rm s}^2\omega^2|}+\frac{\omega^2}{|(|\xi|^2-c_{\rm s}^2\omega^2)(|\xi|^2-c_{\rm p}^2\omega^2)|},
\end{eqnarray*}
we get 
\begin{align*}
|{\rm I}_3|&\lesssim\int_{\Omega_3}\left[\frac{(1+|\xi|^2)^s}{\|\xi|^2-c_{\rm
p}^2\omega^2|}+ \frac{(1+|\xi|^2)^s}{\|\xi|^2-c_{\rm
s}^2\omega^2|}+\frac{\omega^2(1+|\xi|^2)^s}{|(|\xi|^2-c_{\rm
s}^2\omega^2)(|\xi|^2-c_{\rm p}^2\omega^2)|}\right]\big|
\widehat{\mathcal J^{-s}f_1}(\xi) \big|\big| \widehat{\mathcal J^{-s}g_1}(\xi)
\big|d\xi\\
&\lesssim\omega^{-2+2s}\int_{\Omega_3}\big| \widehat{\mathcal J^{-s}f_1}(\xi)
\big|\big| \widehat{\mathcal J^{-s}g_1}(\xi) \big|d\xi\\\label{x7}
&\lesssim \omega^{-2+2s}\|f_1\|_{H^{-s}}\|g_1\|_{H^{-s}},
\end{align*}
where in the second step we have used the following estimates: if $|\xi|<(c_{\rm
p}-\varepsilon_2)\omega$, then
\[
\frac{(1+|\xi|^2)^s}{\|\xi|^2-c_{\rm p}^2\omega^2|}\le\frac{(1+|\xi|^2)^s}{\varepsilon_2\omega(|\xi|+c_{\rm p}\omega)}\le
\frac{(1+|(c_{\rm p}-\varepsilon_2)\omega|^2)^s}{\varepsilon_2c_{\rm
p}\omega^2}\lesssim\omega^{-2+2s};
\]
if $|\xi|>(c_{\rm p}+\varepsilon_2)\omega$ in $\Omega_3$, then 
\[
\frac{(1+|\xi|^2)^s}{\|\xi|^2-c_{\rm p}^2\omega^2|}\le\frac{(1+|\xi|^2)^s}{\varepsilon_2\omega(|\xi|+c_{\rm p}\omega)}\le\frac{(2|\xi|^2)^s}{\varepsilon_2\omega|\xi|}\lesssim\frac1{\omega|\xi|^{1-2s}}\lesssim \omega^{-2+2s}.
\]

For the term ${\rm I}_1$, we have
\begin{align*}
{\rm I}_1 &= c_{\rm p}^2\int_{\Omega_1}\frac{1}{|\xi|^2-c_{\rm
p}^2\omega^2}(1+|\xi|^2)^s\widehat{\mathcal
J^{-s}f_1}(\xi)\widehat{\mathcal J^{-s}g_1} (\xi)d\xi\\
&\quad + c_{\rm s}^2c_{\rm
p}^2\int_{\Omega_1}\frac{(\lambda+\mu)\xi_2^2}{(|\xi|^2-c_{\rm
s}^2\omega^2)(|\xi|^2-c_{\rm
p}^2\omega^2)}(1+|\xi|^2)^s\widehat{\mathcal
J^{-s}f_1}(\xi)\widehat{\mathcal J^{-s}g_1} (\xi)d\xi\\
& =:{\rm I}_{11}+{\rm I}_{12}.
\end{align*}
For $\xi\in\Omega_1$, we can choose $\varepsilon_1$ small enough such that
$\|\xi|-c_{\rm p}\omega|\geq c\omega$ for some $c>0$, and follow similarly the
estimate of ${\rm I}_3$ to get 
 \begin{eqnarray*}\label{x9}
 |{\rm I}_{11}|\lesssim \omega^{-2+2s}\|f_1\|_{H^{-s}}\|g_1\|_{H^{-s}}. 
 \end{eqnarray*}
To estimate ${\rm I}_{12}$, we make the following change of variables:
\begin{eqnarray*}%\label{x10}
\xi^*=\xi+2(c_{\rm s}\omega-|\xi|)\tilde{\xi}=2c_{\rm
s}\omega\tilde{\xi}-\xi,
\end{eqnarray*}
where $\tilde{\xi}:=\xi/|\xi|$. It can be easily verified that the change of
variables maps the domain $\Omega_{11}:=\{\xi\in{\mathbb R^2}:c_{\rm
s}\omega-\varepsilon_1\omega<|\xi|<c_{\rm s}\omega\}$ to the domain 
$\Omega_{12}:=\{\xi\in{\mathbb R^2}:c_{\rm s}\omega<|\xi|<c_{\rm
s}\omega+\varepsilon_1\omega\}$, and the Jacobian for the change of variables
is 
\begin{eqnarray*}%\label{x11}
J(\xi)=\frac{2c_{\rm s}\omega}{|\xi|}-1.
\end{eqnarray*}

Using the fact $\Omega_1=\Omega_{11}\cup
\Omega_{12}\cup\{\xi\in\mathbb{R}^2:|\xi|=c_{\rm s}\omega\}$ with
$\{\xi\in\mathbb{R}^2:|\xi|=c_{\rm s}\omega\}$ being a set of zero measure, we
obtain 
\begin{align*}
{\rm I}_{12}&=c_{\rm s}^2c_{\rm p}^2\int_{\Omega_{11}\cup
\Omega_{12}}\frac{(\lambda+\mu)\xi_2^2}{(|\xi|^2-c_{\rm
s}^2\omega^2)(|\xi|^2-c_{\rm
p}^2\omega^2)}(1+|\xi|^2)^s\widehat{\mathcal
J^{-s}f_1}(\xi)\widehat{\mathcal J^{-s}g_1} (\xi)d\xi\\
&=c_{\rm s}^2c_{\rm
p}^2\int_{\Omega_{12}}\frac{(\lambda+\mu)\xi_2^2}{(|\xi|^2-c_{\rm
s}^2\omega^2)(|\xi|^2-c_{\rm
p}^2\omega^2)}(1+|\xi|^2)^s\widehat{\mathcal
J^{-s}f_1}(\xi)\widehat{\mathcal J^{-s}g_1} (\xi)d\xi\\
&\quad +c_{\rm s}^2c_{\rm
p}^2\int_{\Omega_{12}}\frac{(\lambda+\mu){\xi^*_2}^2}{(|\xi^*|^2-c_{\rm
s}^2\omega^2)(|\xi^*|^2-c_{\rm
p}^2\omega^2)}(1+|\xi^*|^2)^s\widehat{\mathcal
J^{-s}f_1}(\xi^*)\widehat{\mathcal J^{-s} g_1 } (\xi^*)J(\xi)d\xi\\
&=c_{\rm s}^2c_{\rm
p}^2\int_{\Omega_{12}}m_1(\xi,\omega)(1+|\xi|^2)^s\widehat{\mathcal J^{-s}f_1}
(\xi)\widehat{\mathcal J^{-s}g_1}(\xi)d\xi\\
&\quad+c_{\rm s}^2c_{\rm
p}^2\int_{\Omega_{12}}m_2(\xi,\omega)\big[(1+|\xi^*|^2)^s\widehat{\mathcal
J^{-s}
f_1}(\xi^*)\widehat{\mathcal
J^{-s}g_1}(\xi^*)-(1+|\xi|^2)^s\widehat{\mathcal J^{-s}f_1}
(\xi)\widehat{\mathcal J^{-s}g_1}(\xi)\big]J(\xi)d\xi\\\label{x12}
&=:{\rm I}_{13}+{\rm I}_{14},
\end{align*}
where
\begin{align*}
m_1(\xi,\omega)&=\frac{(\lambda+\mu)\xi_2^2}{(|\xi|^2-c_{\rm
s}^2\omega^2)(|\xi|^2-c_{\rm
p}^2\omega^2)}+\frac{(\lambda+\mu){\xi^*_2}^2}{(|\xi^*|^2-c_{\rm
s}^2\omega^2)(|\xi^*|^2-c_{\rm p}^2\omega^2)}J(\xi),\\
m_2(\xi,\omega)&=\frac{(\lambda+\mu){\xi^*_2}^2}{(|\xi^*|^2-c_{\rm s}^2\omega^2)(|\xi^*|^2-c_{\rm p}^2\omega^2)}.
\end{align*}
For $\xi\in\Omega_{12}$, it is not difficult to show that $\xi^*\in\Omega_{11}$
with $|\xi^*|^2=4c_{\rm s}^2\omega^2+|\xi|^2-4c_{\rm s}\omega|\xi|$. Then there
exists a constant $C>0$ such that 
\begin{eqnarray*}
\left|\frac{(\lambda+\mu)\xi_2^2}{|\xi|^2-c_{\rm p}^2\omega^2}\right|\leq C
\end{eqnarray*}
and
\begin{align*}
\left|\frac{1}{|\xi|^2-c_{\rm s}^2\omega^2}+\frac{1}{|\xi^*|^2-c_{\rm s}^2\omega^2}J(\xi)\right|
&=\left|\frac1{(|\xi|-c_{\rm s}\omega)(|\xi|+c_{\rm s}\omega)}+\frac1{(|\xi|-c_{\rm s}\omega)(|\xi|-3c_{\rm s}\omega)}\frac{2c_{\rm s}\omega-|\xi|}{|\xi|}\right|\\
&=\frac{2c_{\rm s}^2\omega^2-2c_{\rm s}\omega|\xi|}{(|\xi|-c_{\rm s}\omega)(|\xi|+c_{\rm s}\omega)(|\xi|-3c_{\rm s}\omega)|\xi|}
\lesssim\omega^{-2},
\end{align*}
which leads to 
\begin{align*}
|m_1(\xi,\omega)|&\leq \left|\frac{(\lambda+\mu)\xi_2^2}{|\xi|^2-c_{\rm p}^2\omega^2}\left[\frac{1}{|\xi|^2-c_{\rm s}^2\omega^2}+\frac{1}{|\xi^*|^2-c_{\rm s}^2\omega^2}J(\xi)\right]\right|\\
&\quad +\left|\frac{J(\xi)}{|\xi^*|^2-c_{\rm s}^2\omega^2}\left[\frac{(\lambda+\mu){\xi^*_2}^2}{|\xi^*|^2-c_{\rm p}^2\omega^2}-\frac{(\lambda+\mu)\xi_2^2}{|\xi|^2-c_{\rm p}^2\omega^2}\right]\right|
\lesssim\omega^{-2}.
\end{align*}
Hence, the term ${\rm I}_{13}$ admits the estimate 
\begin{eqnarray*}\label{x15}
|{\rm I}_{13}|\lesssim\omega^{-2+2s}\|f_1\|_{H^{-s}}\|g_1\|_{H^{-s}}.
\end{eqnarray*}
The item ${\rm I}_{14}$ can be decomposed as
\begin{align*}
{\rm I}_{14}&=c_{\rm s}^2c_{\rm
p}^2\int_{\Omega_{12}}m_2(\xi,\omega)\left[(1+|\xi^*|^2)^s-(1+|\xi|^2)^s\right]
\widehat{\mathcal
J^{-s}f_1}(\xi)\widehat{\mathcal J^{-s}g_1}(\xi)J(\xi)d\xi\\
&\quad +c_{\rm s}^2c_{\rm
p}^2\int_{\Omega_{12}}m_2(\xi,\omega)(1+|\xi^*|^2)^s\big[\widehat{\mathcal
J^{-s} f_1 } (\xi^*)-\widehat{\mathcal
J^{-s}f_1}(\xi)\big]\widehat{\mathcal J^{-s}g_1} (\xi)J(\xi)d\xi\\
&\quad +c_{\rm s}^2c_{\rm
p}^2\int_{\Omega_{12}}m_2(\xi,\omega)(1+|\xi^*|^2)^s\big[\widehat{\mathcal
J^{-s} g_1 } (\xi^*)-\widehat{\mathcal
J^{-s}g_1}(\xi)\big]\widehat{\mathcal J^{-s}f_1} (\xi^*)J(\xi)d\xi\\
&=:{\rm I}_{15}+{\rm I}_{16}+{\rm I}_{17}.
\end{align*}
By the mean value theorem, we get 
\begin{align*}%\label{x17}
&\left|m_2(\xi,\omega)\left[(1+|\xi^*|^2)^s-(1+|\xi|^2)^s\right]J(\xi)\right|\\
&=\left|\frac{(\lambda+\mu)(\xi_2^*)^2}{(|\xi^*|^2-c_{\rm s}^2\omega^2)(|\xi^*|^2-c_{\rm p}^2\omega^2)}s\left(1+\theta|\xi^*|^2+(1-\theta)|\xi|^2\right)^{s-1}(|\xi^*|^2-|\xi|^2)\frac{2c_{\rm s}\omega-|\xi|}{|\xi|}\right|\\
&=\left|\frac{(\lambda+\mu)(\xi_2^*)^2}{(|\xi|-c_{\rm s}\omega)(|\xi|-3c_{\rm s}\omega)(|\xi^*|^2-c_{\rm p}^2\omega^2)}s\left(1+\theta|\xi^*|^2+(1-\theta)|\xi|^2\right)^{s-1}4c_{\rm s}\omega(c_{\rm s}\omega-|\xi|)\frac{2c_{\rm s}\omega-|\xi|}{|\xi|}\right|\\
&\lesssim\omega^{-2+2s}
\end{align*}
with some $\theta\in(0,1)$. It shows that 
\begin{eqnarray*}\label{x18}
|{\rm I}_{15}|\lesssim\omega^{-2+2s}\|f_1\|_{H^{-s}}\|g_1\|_{H^{-s}}.
\end{eqnarray*}

To estimate ${\rm I}_{16}$ and ${\rm I}_{17}$, we employ the following
characterization of $W^{1,p}({\mathbb R^d})$ introduced in \cite{PH}.

\begin{lemma}\label{HLL}
For $1<p\leq\infty$, the function $u\in W^{1,p}({\mathbb R^d})$ if and only if
there exist $g\in L^p({\mathbb R^d})$ and $C>0$ such that
\begin{eqnarray*}%\label{x19}
|u(x)-u(y)|\leq C|x-y|(g(x)+g(y)).
\end{eqnarray*}
Moreover, we can choose $g=M(|\nabla u|)$, where $M$ is defined by
\begin{eqnarray*}%\label{x20}
M(f)(x)=\sup_{r>0}\frac{1}{|B(x,r)|}\int_{B(x,r)}|f(y)|dy
\end{eqnarray*}
and is called the Hardy--Littlewood maximal function of $f$.
\end{lemma}

For $f_1\in C_0^{\infty}$, we have $\widehat{\mathcal J^{-s}f_1}\in
\mathcal{S}\subset H^1$. An application of Lemma \ref{HLL} gives
\begin{equation}\label{x21}
\left|\widehat{\mathcal
J^{-s}f_1}(\xi^*)-\widehat{\mathcal J^{-s}f_1}(\xi)\right|\lesssim \big|c_{\rm
s}\omega-|\xi|\big|\big[M(|\nabla \widehat{\mathcal
J^{-s}f_1}|)(\xi^*)+M(|\nabla \widehat{\mathcal J^{-s}f_1}|)(\xi)\big].
\end{equation} 
By \cite[Theorem 2.1]{LP}, we get
\begin{align}\label{x23}
\|M(|\nabla \widehat{\mathcal J^{-s}f_1}|)\|_{L^2}&\lesssim \|\nabla
\widehat{\mathcal J^{-s}f_1}\|_{L^2}\lesssim
\|(I-\Delta)^{\frac12}\widehat{\mathcal J^{-s}f_1}\|_{L^2}\notag\\
&=\|(I-\Delta)^{\frac12}(1+|\cdot|^2)^{-\frac
s2}\hat{f_1}(\cdot)\|_{L^2}=\|f_1\|_{H^{-s}_1},
\end{align}
where \eqref{eq:Hsdelta} is used in the last step. Combining (\ref{x21}) and
(\ref{x23}) gives 
\begin{align*}%\label{x24}
|{\rm I}_{16}|&\lesssim \omega^{-1+2s}\int_{\Omega_{12}}\big[M(|\nabla
\widehat{\mathcal J^{-s}f_1}|)(\xi^*)+M(|\nabla
\widehat{\mathcal
J^{-s}f_1}|)(\xi)\big]|\widehat{\mathcal J^{-s}g_1}(\xi)|d\xi\\
&\lesssim \omega^{-1+2s}\|f_1\|_{H_1^{-s}}\|g_1\|_{H^{-s}}.
\end{align*}
The item ${\rm I}_{17}$ can be similarly estimated and satisfies
\begin{eqnarray*}%\label{x25}
|{\rm I}_{17}|\lesssim \omega^{-1+2s}\|f_1\|_{H^{-s}}\|g_1\|_{H_1^{-s}}.
\end{eqnarray*}

Hence we conclude from the above estimates that  
\begin{eqnarray*}%\label{x26}
|{\rm I}_{1}|\lesssim \omega^{-1+2s}\|f_1\|_{H_1^{-s}}\|g_1\|_{H_1^{-s}}.
\end{eqnarray*}
Similarly, we may repeat the steps for the estimate of $I_1$ and show that 
\begin{eqnarray*}%\label{x27}
|{\rm I}_{2}|\lesssim \omega^{-1+2s}\|f_1\|_{H_1^{-s}}\|g_1\|_{H_1^{-s}}.
\end{eqnarray*}
It follows from the estimates of ${\rm I}_j$, $j=1,2,3$ that \eqref{x3}
satisfies the estimate
\begin{eqnarray*}%\label{x28}
\left|\int_{\mathbb R^2}\widehat{G_{11}}(\xi)\hat{f}_1(\xi)\hat{g}_1(\xi)d\xi\right|\lesssim
\omega^{-1+2s}\|f_1\|_{H_1^{-s}}\|g_1\|_{H_1^{-s}}.
\end{eqnarray*}

Next is to estimate \eqref{x4}. Let
\begin{align*}
{\rm II}_j&=-c_{\rm s}^2c_{\rm
p}^2\int_{\Omega_j}\frac{(\lambda+\mu)\xi_1\xi_2}{(|\xi|^2-c_{\rm
s}^2\omega^2)(|\xi|^2-c_{\rm
p}^2\omega^2)}\hat{f}_2(\xi)\hat{g}_1(\xi)d\xi\\
&=-c_{\rm s}^2c_{\rm
p}^2\int_{\Omega_j}\frac{(\lambda+\mu)\xi_1\xi_2}{(|\xi|^2-c_{\rm
s}^2\omega^2)(|\xi|^2-c_{\rm
p}^2\omega^2)}(1+|\xi|^2)^s\widehat{\mathcal
J^{-s}f_2}(\xi)\widehat{\mathcal J^{-s}g_1} (\xi)d\xi ,\quad j=1,2,3.
\end{align*}

Following the same estimate as that of ${\rm I}_3$ and noting 
\[
|(\lambda+\mu)\xi_1\xi_2|=|(\lambda+2\mu)\xi_1\xi_2-\mu\xi_1\xi_2|\le\frac{(\lambda+2\mu)|\xi|^2-\omega^2}2+\frac{\mu|\xi|^2-\omega^2}2+\omega^2,
\]
we have
\[
|{\rm II}_3|\lesssim\omega^{-2+2s}\|f_2\|_{H^{-s}}\|g_1\|_{H^{-s}}. 
\]
As for the estimate of
\[
{\rm II}_1=-c_{\rm s}^2c_{\rm
p}^2\int_{\Omega_1}\frac{(\lambda+\mu)\xi_1\xi_2}{(|\xi|^2-c_{\rm
s}^2\omega^2)(|\xi|^2-c_{\rm
p}^2\omega^2)}(1+|\xi|^2)^s\widehat{\mathcal
J^{-s}f_2}(\xi)\widehat{\mathcal J^{-s}g_1} (\xi)d\xi,
\]
it is similar to that of ${\rm I}_{12}$ and admits
\[
|{\rm II}_1|\lesssim\omega^{-1+2s}\|f_2\|_{H^{-s}_1}\|g_1\|_{H^{-s}_1},
\]
which may further lead to the estimate
\[
|{\rm II}_2|\lesssim\omega^{-1+2s}\|f_2\|_{H^{-s}_1}\|g_1\|_{H^{-s}_1}. 
\]
Combining the above estimates yields 
\[
\left|\int_{\mathbb R^2}\widehat{G_{11}}(\xi)\hat{f}_1(\xi)\hat{g}_1(\xi)d\xi\right|\lesssim\omega^{
-1+2s}\|f_2\|_{H^{-s}_1}\|g_1\|_{H^{-s}_1}. 
\]

It follows from  \eqref{x3}--\eqref{x4} that 
(\ref{x2}) has the following estimate: 
\begin{eqnarray*}%\label{x29}
\left|\langle H_{\omega}{\bm f}, {\bm g}\rangle\right|\lesssim
\omega^{-1+2s}\|{\bm f}\|_{\bm H_1^{-s}}\|{\bm g}\|_{\bm
H_1^{-s}}\quad\forall\,\bm f,\bm g\in\bm C_0^\infty.
\end{eqnarray*}
This result can be extended for any $\bm f,\bm g\in\bm H_1^{-s}$ since $\bm
C_0^\infty$ is dense in $\bm H_1^{-s}$. The density argument can be found in
\cite[Theorem 2.2]{LLM}. It then completes the proof of \eqref{y1}.

To prove (\ref{y2}), let ${\bm f}=(f_1,f_2)^\top\in \bm C_0^{\infty}$. We have 
\begin{align}\notag
(H_{\omega}{\bm f})(x)&=\int_{\mathbb R^2}{\bm G}(x,y){\bm f}(y)dy\\\notag
&=\int_{\mathbb R^2}(1+|\xi|^2)^{\frac{s}{2}}\widehat{\bm
G}(x,\xi)\widehat{\mathcal J^{-s}{\bm f}}(\xi)d\xi\\\label{x32}
&=\int_{\mathbb R^2}(1+|\xi|^2)^{\frac{s}{2}}
\begin{bmatrix} \widehat{G_{11}}(x,\xi)\widehat{\mathcal J^{-s}f_1}(\xi)+\widehat{G_{12}}(x,\xi)\widehat{\mathcal J^{-s}f_2}(\xi)\\
\widehat{G_{21}}(x,\xi)\widehat{\mathcal J^{-s}f_1}(\xi)+\widehat{G_{22}}(x,\xi)\widehat{\mathcal J^{-s}f_2}(\xi)
\end{bmatrix}
d\xi,
\end{align}
where $\widehat{\bm G}(x,\xi)$, different from $\widehat{\bm G}(\xi)$, 
denotes the Fourier transform of $\bm G(x,y)$ obtained by taking the Fourier
transform on both sides of \eqref{x31} with respect to $y$ and satisfies
\[
-\mu|\xi|^2\widehat{\bm G}(x,\xi)-(\lambda+\mu)\xi\cdot\xi^\top\widehat{\bm G}(x,\xi)+\omega^2\widehat{\bm G}(x,\xi)=-e^{-{\rm i}x\cdot\xi}\bm I.
\]
Comparing the above equation with (\ref{x1}), we get $\widehat{G_{ij}}(x,\xi)=e^{-{\rm i}x\cdot\xi}\widehat{G_{ij}}(\xi)$. It follows
from the same arguments as those for the item (\ref{x3}) that  
\begin{align*}
&\left|\int_{\mathbb R^2}(1+|\xi|^2)^{\frac{s}{2}}\widehat{G_{11}}(x,\xi)\widehat{\mathcal J^{-s}f_1}(\xi)d\xi\right|\\
&=\left|\int_{\mathbb R^2} \widehat{G_{11}}(\xi)(1+|\xi|^2)^{\frac{s+\epsilon+1}{2}}\widehat{\mathcal J^{-s}f_1}(\xi)e^{-{\rm
i}x\cdot\xi}(1+|\xi|^2)^{\frac{-1-\epsilon}2}d\xi\right|\\
&\lesssim\frac1{\omega^{1-2\frac{s+\epsilon+1}{2}}}\|f_1\|_{H^{-s}_1}
\lesssim\omega^{s+\epsilon}\|f_1\|_{H^{-s}_1},
\end{align*}
where we have used the fact that the function
\[
g(\xi):=e^{-{\rm i}x\cdot\xi}(1+|\xi|^2)^{\frac{-1-\epsilon}2}
\]
satisfies $g\in H^1$ for any $x\in\mathcal V$. The estimates can be similarly
obtained for the other three items in (\ref{x32}). Therefore we have 
\begin{eqnarray*}
\left\|H_{\omega}{\bm f}\right\|_{\bm L^\infty(\mathcal V)}\lesssim \omega^{s+\epsilon}\|{\bm f}\|_{\bm H^{-s}_1},
\end{eqnarray*}
which completes the proof of (\ref{y2}).
\end{proof}

Based on the estimates for the operator $H_\omega$, the following results
present the estimates for the operator $K_\omega$. 

\begin{lemma}\label{lm:K}
Let $\mathcal V\subset\mathbb R^2$ be a bounded domain and $\rho$ satisfy Assumption
\ref{as:rho}. For any $s\in(\frac{2-m}2,\frac12)$, it holds almost surely that 
\begin{align*}
\|K_{\omega}\|_{\mathcal{L}( \bm H_{-1}^{s})} &\lesssim
\omega^{-1+2s},\\
\|K_{\omega}\|_{\mathcal{L}(\bm H_{-1}^{s}, \bm L^{\infty}(\mathcal V))} &\lesssim
\omega^{s+\epsilon}.
\end{align*}
\end{lemma}

\begin{proof}
For any $\bm u\in\bm H^{s}_{-1}$, it holds $K_\omega\bm u=H_\omega
(\rho\bm u)$ with $H_\omega$ being a bounded operator from $\bm
H_1^{-s}$ to $\bm H_{-1}^s$ according to Lemma \ref{lm:H}. 

We first claim that $\rho\bm u\in\bm H_1^{-s}$ for any $\bm u\in\bm H_{-1}^s$.
Note that $\rho\in
W^{-\gamma,p}_0(D)$ for any $\gamma>\frac{2-m}2$ and $p>1$ according to Lemma \ref{le1}.
For any $\bm u,\bm v\in\mathcal{S}$, define
$\langle \rho\bm u,\bm v\rangle:=\langle \rho,\bm
u\cdot\bm v\rangle$ and a cutoff function $\vartheta\in C_0^{\infty}$ whose
support $\tilde D$ has a locally Lipschitz boundary and $\vartheta(x)=1$ if
$x\in D\subset\tilde D$. It is easy to verify that 
\begin{align*}
\left|\langle \rho\bm u,\bm v\rangle\right|&=\left|\langle
\rho,(\vartheta\bm u)\cdot(\vartheta\bm v)\rangle\right|\\
&=\left|\langle(I-\Delta)^{-\gamma}
\rho,(I-\Delta)^\gamma[(\vartheta\bm u)\cdot(\vartheta\bm
v)]\rangle\right|\\
&\le \|\rho\|_{W^{-\gamma,p}}\|(I-\Delta)^\gamma[(\vartheta\bm
u)\cdot(\vartheta\bm v)]\|_{L^{q}}
\end{align*}
with $q>1$ satisfying $\frac1p+\frac1q=1$.
It follows from  the fractional Leibniz principle with $\tilde q$ satisfying $\frac1{q}=\frac12+\frac1{\tilde q}$ that 
\[
\|(I-\Delta)^\gamma[(\vartheta\bm u)\cdot(\vartheta\bm v)]\|_{L^{q}}\le\|\vartheta\bm u\|_{\bm L^2(\tilde D)}\|\vartheta\bm v\|_{\bm
W^{\gamma,\tilde q}(\tilde D)}+\|\vartheta\bm v\|_{\bm L^2(\tilde
D)}\|\vartheta\bm u\|_{\bm W^{\gamma,\tilde q}(\tilde D)}.
\]
For any $s\in(\frac{2-m}2,\frac12)$, there exist $\gamma\in(\frac{2-m}2,s)$ and $q>1$ such that $\gamma\le s$ and $\frac1q-\frac12=\frac1{\tilde
q}>\frac12-\frac{s-\gamma}2$, which implies $\bm H^s(\tilde D)\hookrightarrow\bm
W^{\gamma,\tilde q}(\tilde D)$. Hence
\[
\left|\langle \rho\bm u,\bm
v\rangle\right|\le\|\rho\|_{W^{-\gamma,p}}\|\vartheta\bm
u\|_{\bm W^{\gamma,\tilde q}(\tilde D)}\|\vartheta\bm v\|_{\bm W^{\gamma,\tilde q}(\tilde D)}\lesssim\|\rho\|_{W^{-\gamma,p}}\|\vartheta\bm u\|_{\bm H^s(\tilde D)}\|\vartheta\bm v\|_{\bm H^s(\tilde D)}.
\]
Using the facts that $\|\vartheta\bm u\|_{\bm H^s(\tilde D)}\lesssim\|\bm
u\|_{\bm H^s_{-2}}\lesssim\|\bm u\|_{\bm H^s_{-1}}$ and that $\mathcal{S}$ is
dense in $\bm H^s_{-1}$ proved in \cite[Theorem 2.2]{LLM}, we get almost
surely that 
\begin{align*}
\left|\langle \rho\bm u,\bm
v\rangle\right|\lesssim\|\bm u\|_{\bm H^s_{-1}}\|\bm v\|_{\bm
H^s_{-1}}\quad\forall\,\bm u,\bm v\in\bm H_{-1}^s,
\end{align*}
which completes the claim. Then the following two estimates hold almost surely: 
\[
\|K_\omega\bm u\|_{\bm H^s_{-1}}\le\|H_\omega\|_{\mathcal{L}(\bm H_1^{-s},\bm
H^s_{-1})}\|\rho\bm u\|_{\bm
H^{-s}_1}\lesssim\omega^{-1+2s}\|\bm u\|_{\bm H^s_{-1}}
\]
and
\[
\|K_\omega\bm u\|_{\bm L^{\infty}(\mathcal V)}\le\|H_\omega\|_{\mathcal{L}(\bm H_1^{-s},\bm
L^{\infty}(\mathcal V))}\|\rho\bm u\|_{\bm
H^{-s}_1}\lesssim\omega^{s+\epsilon}\|\bm u\|_{\bm
H^s_{-1}},
\]
which complete the proof.
\end{proof}

\subsection{Convergence of the Born series}

Let assumptions in Lemma \ref{lm:K} hold and
$U\subset\mathbb{R}^2\backslash\overline{D}$ be a bounded and convex measurement domain
which has a locally Lipschitz boundary and a positive distance from $D$. This
section is to show the convergence of the Born series defined in \eqref{c2}. 

It follows from (\ref{c2}) that 
\begin{eqnarray}\label{c4}
[(I+K_{\omega})\sum_{j=0}^{N}\bm{u}_j(\cdot,y)](x)=\bm{u}_0(x,y)+(-1)^{N}[K_{\omega}^{N+1}\bm{u}_0(\cdot,y)](x).
\end{eqnarray}
Note that 
\[
[K_\omega\bm u_0(\cdot,y)](x)=\int_{D}\bm
G(x,z,\omega)\rho(z)\bm u_0(z,y)dz\quad \forall\, x,y\in U,
\]
where
$\bm{u}_0(z,y)=\bm{G}(z,y,\omega)\bm{a}$ and $\bm{G}(z,y,\omega)$ is smooth
for any $z\in D$ and $y\in U$. 

We begin with the estimate for $\bm u_0$.

\begin{lemma}\label{lm:u0}
Let $U\subset\mathbb R^2\backslash\overline{D}$ be a bounded and convex domain having a
locally Lipschitz boundary and a positive distance to $D$.
For any $s\in[0,1]$, $p\in(1,\infty)$ and any fixed $y\in U$, the following estimate holds:
\[
\|\bm u_0(\cdot,y)\|_{\bm W^{s,p}(D)}\lesssim\omega^{-\frac12+s}.
\]
\end{lemma}

\begin{proof}
For any $y\in U$, it is easy to check that 
\begin{align*}
\|\bm u_0(\cdot,y)\|_{\bm L^p(D)}&=\|\bm G(\cdot,y,\omega)\bm a\|_{\bm L^p(D)}\lesssim\omega^{-\frac12},\\
\|\bm u_0(\cdot,y)\|_{\bm W^{1,p}(D)}&=\|\bm G(\cdot,y,\omega)\bm a\|_{\bm W^{1,p}(D)}\lesssim\omega^{\frac 12}.
\end{align*}
Utilizing the interpolation inequality \cite{L3}, we get
\begin{align*}
\|\bm u_0(\cdot,y)\|_{\bm
W^{s,p}(D)}\lesssim\|\bm u_0(\cdot,y)\|_{\bm L^p(D)}^{1-s}\|\bm u_0(\cdot,y)\|_{\bm
W^{1,p}(D)}^s\lesssim\omega^{-\frac12+s},
\end{align*}
which completes the proof.
\end{proof}

By Lemmas \ref{lm:K} and \ref{lm:u0}, we have for $s\in(\frac{2-m}2,\frac12)$ that
\begin{eqnarray*}\notag
\|K_{\omega}^{N+1}\bm{u}_0\|_{\bm H_{-1}^s(U)}\lesssim
\omega^{(-1+2s)(N+1)}\|\bm{u}_0(\cdot,y)\|_{\bm H^s_{-1}(D)}\lesssim\omega^{(-1+2s)(N+1)-\frac{1}{2}+s}\to
0
\end{eqnarray*}
as $N\to\infty$, where we use the fact 
\[
\|\bm u_0(\cdot,y)\|_{\bm
H^s_{-1}(D)}=\|(1+|\cdot|^2)^{-\frac12}(I-\Delta)^{\frac s2}\bm
u_0(\cdot,y)\|_{\bm L^2(D)}\lesssim\|\bm u_0(\cdot,y)\|_{\bm
H^s(D)}.
\] 
Combining the above estimate with (\ref{c4}) leads to
\begin{eqnarray*}%\label{c6}
(I+K_{\omega})\sum_{j=0}^{\infty}\bm{u}_j=\bm{u}_0\quad\text{in}\quad\bm H_{-1}^s(U).
\end{eqnarray*}

Note also that $\bm{G}(\cdot,y,\omega)\in (L^2_{\rm
loc}\cap W_{\rm loc}^{1,p'})^{2\times 2}$ for any
$p'\in (1,2)$. Choosing 
$p'=2-\epsilon$ for sufficient small
$\epsilon>0$, we may follow the same proof as that of Theorem
\ref{tm:LS} and get $\bm{W}^{1,p'}(U)\hookrightarrow\bm H^s(U)$, which implies
that $\bm u_0(\cdot,y)\in\bm H^s(U)\hookrightarrow\bm W^{\gamma,q}(U)$
and $(I+K_{\omega})^{-1}\bm{u}_0=\bm{u}$ in $\bm W^{\gamma,q}(U)$. Hence, the
Born series converges to the unique solution $\bm u$ of \eqref{a1} in $\bm
W^{\gamma,q}(U)$ and 
\begin{eqnarray}\label{c8}
\bm{u}=\sum_{j=0}^{\infty}\bm{u}_j.
\end{eqnarray}
Moreover, 
\begin{align}\label{eq:uj}
\|\bm u-\sum_{j=0}^N\bm u_j\|_{\bm L^{\infty}(U)}&\lesssim \sum_{j=N+1}^\infty\|K_\omega^{j}\bm u_0\|_{\bm L^\infty(U)}\notag\\
&\le\sum_{j=N+1}^\infty\|K_\omega\|_{\mathcal{L}(\bm H_{-1}^s,\bm L^\infty(U))}\|K_\omega\|_{\mathcal{L}(\bm H_{-1}^s)}^{j-1}\|\bm u_0(\cdot,y)\|_{\bm H_{-1}^s(D)}\notag\\
&\lesssim \omega^{s+\epsilon+(-1+2s)N-\frac12+s}\to0
\end{align}
as $N\to\infty$, which implies that the equation \eqref{c8} also holds in $\bm
L^{\infty}(U)$.

\section{The inverse scattering problem}
\label{sec:5}

In this section, we study the inverse scattering problem which is to reconstruct
the micro-correlation strength $\phi$ of the random potential $\rho$. 

We consider the case $y=x$ and recall that the notations $\bm{u}^s(x,\omega,\bm{a})$ and $\bm{u}_j(x,\omega,\bm{a})$ stand for $\bm{u}^s(x,x,\omega,\bm{a})$ and $\bm{u}_j(x,x,\omega,\bm{a})$, respectively.
Then we rewrite \eqref{c8} in terms of the scattered field
\[
\bm u^s(x,\omega,{\bm a})={\bm u}_1(x,\omega,{\bm a})+\bm u_2(x,\omega,\bm a)+{\bm b}(x,\omega,{\bm a}),
\]
where 
\[
\bm b(x,\omega,{\bm a})=\sum_{j=3}^\infty\bm u_j(x,\omega,{\bm a}). 
\]

\subsection{The analysis of $\bm{u}_1$}

This subsection is devoted to the analysis of the leading term $\bm{u}_1$.
Explicitly, it is given by
\begin{eqnarray}\label{d1}
\bm{u}_1(x,\omega,\bm{a})=-\int_D\rho(z)\bm{G}(x,z,\omega)^2\bm{a}dz.
\end{eqnarray}

\begin{theorem}\label{thm3}
Let $\rho$ satisfy Assumption \ref{as:rho} and $U\subset\mathbb R^2\backslash\overline{D}$ be a bounded and convex domain having a locally Lipschitz boundary and a positive distance to $D$.
Then for all $x\in U$, it holds 
\begin{eqnarray*}%\label{d2}
\lim_{Q\to\infty}\frac{1}{Q-1}\int_1^Q\omega^{m+2}\sum_{j=1}^2|\bm{u}_1
(x,\omega,\bm{a}_j)|^2d\omega
=\frac{c_{\rm s}^{6-m}+c_{\rm p}^{6-m}}{2^{m+6}\pi^2}\int_{\mathbb R^2}\frac{1}{|x-\zeta|^{2}}\phi(\zeta)d\zeta\quad a.s.,
\end{eqnarray*}
where $\bm{a}_1$ and $\bm{a}_2$ are two orthonormal vectors in $\mathbb{R}^2$. 
\end{theorem}

Before giving the proof of Theorem \ref{thm3}, we first introduce the truncation
of the Green tensor $\bm G$ and some a priori estimates. 
Let $H_{n}^{(1)}$ be the Hankel function of the first kind with order $n$,
which has the following asymptotic expansion (cf. \cite{AS}): 
\begin{align}\label{d3}
H_{n}^{(1)}(c)=\sum_{j=0}^N b_j^{(n)}c^{-(j+\frac12)}e^{{\rm i}(c-\frac{n\pi}2)}+O(|c|^{-N-\frac{3}{2}}),\quad c\in\mathbb{C},~|c|\to\infty,
\end{align}
where $b_0^{(n)}=\frac{1+\rm i}{\sqrt{\pi}}$ and
\[
b_j^{(n)}=\frac{(1+{\rm i}){\rm i}^j}{\sqrt{\pi}8^jj!}\prod_{l=1}^j\left(4n^2-(2l-1)^2\right),\quad j\ge1. 
\]
For the sufficiently large argument $c=\kappa |z|$, define the truncated Hankel
function 
\begin{eqnarray*}%\label{d4}
H_{n,N}^{(1)}(c):=\sum_{j=0}^N b_j^{(n)}c^{-(j+\frac12)}e^{{\rm i}(c-\frac{n\pi}2)}.
\end{eqnarray*}
It follows from (\ref{d3}) that
\begin{eqnarray}\label{d5}
&&\big|H_n^{(1)}(\kappa |z|)-H_{n,N}^{(1)}(\kappa |z|)\big|\lesssim \kappa^{-N-\frac{3}{2}}|z|^{-N-\frac{3}{2}},\\\label{addd5}
&&\big|\nabla_{z}\big[H_n^{(1)}(\kappa |z|)-H_{n,N}^{(1)}(\kappa
|z|)\big]\big|\lesssim\kappa^{-N-\frac{1}{2}}|z|^{-N-\frac{3}{2}}.
\end{eqnarray}
By \eqref{a4}, a straightforward calculation shows that the Green tensor
$\bm{G}$ can be rewritten as 
\begin{align}\label{d6}
\bm{G}(x,y,\omega)&=\left\{\frac{{\rm i}}{4\mu}H_0^{(1)}(\kappa_{\rm
s}|x-y|)-\frac{{\rm i}}{4\omega^2}\frac{1}{|x-y|}\big[\kappa_{\rm
s}H_{1}^{(1)}(\kappa_{\rm s}|x-y|)-\kappa_{\rm
p}H_{1}^{(1)}(\kappa_{\rm p}|x-y|)\big]\right\}\bm{I}\nonumber\\
&\quad+\frac{{\rm
i}}{4\omega^2}\frac{1}{|x-y|^2}\big[\kappa_{\rm
s}^2H_{2}^{(1)}(\kappa_{\rm s}|x-y|)-\kappa_{\rm
p}^2H_{2}^{(1)}(\kappa_{\rm p}|x-y|)\big](x-y)
(x-y)^\top,
\end{align}
where $x-y=(x_1-y_1,x_2-y_2)^\top$. Denote by $\bm{G}^{(N)}$ the truncation
of the Green tensor $\bm G$. Explicitly, 
\begin{align}\label{d8}
\bm{G}^{(N)}(x,y,\omega)&=\left\{\frac{{\rm i}}{4\mu}H_{0,N}^{(1)}(\kappa_{\rm
s}|x-y|)-\frac{{\rm i}}{4\omega^2}\frac{1}{|x-y|}\big[\kappa_{\rm
s}H_{1,N}^{(1)}(\kappa_{\rm s}|x-y|)-\kappa_{\rm
p}H_{1,N}^{(1)}(\kappa_{\rm p}|x-y|)\big]\right\}\bm{I}\nonumber\\
&\quad +\frac{{\rm
i}}{4\omega^2}\frac{1}{|x-y|^2}\big[\kappa_{\rm
s}^2H_{2,N}^{(1)}(\kappa_{\rm s}|x-y|)-\kappa_{\rm
p}^2H_{2,N}^{(1)}(\kappa_{\rm p}|x-y|)\big](x-y)
(x-y)^\top.
\end{align}
Let $G_{ij}$ and $G^{(N)}_{ij}$ be the $(i,j)$-entry 
of $\bm{G}$ and $\bm{G}^{(N)}$, respectively. Using 
(\ref{d3})--(\ref{addd5}), we have the following asymptotic estimates 
\begin{equation}\label{d10}
\begin{aligned}
&|G_{ij}(x,y,\omega)|\lesssim
\omega^{-\frac{1}{2}}|x-y|^{-\frac{1}{2}},\quad\;\;
|\nabla_{x}G_{ij}(x,y,\omega)|\lesssim
\omega^{\frac{1}{2}}|x-y|^{-\frac{1}{2}},\\
&|G^{(N)}_{ij}(x,y,\omega)|\lesssim
\omega^{-\frac{1}{2}}|x-y|^{-\frac{1}{2}},\quad |\nabla_{
x}G^{(N)}_{ij}(x,y,\omega)|\lesssim
\omega^{\frac{1}{2}}|x-y|^{-\frac{1}{2}},\\
&|G_{ij}(x,y,\omega)-G^{(N)}_{ij}(x,y,\omega)|\lesssim \omega^{-N-\frac{3}{2}}|x-y|^{-N-\frac{3}{2}},\\
&|\nabla_{ x}(G_{ij}(x,y,\omega)-G^{(N)}_{ij}(x,y,\omega))|\lesssim \omega^{-N-\frac{1}{2}}|x-y|^{-N-\frac{3}{2}}.
\end{aligned}
\end{equation} 
Replacing $\bm{G}$ by $\bm{G}^{(2)}$ in (\ref{d1}), we define 
\begin{eqnarray}\label{d15}
\bm{u}_1^{(2)}(x,\omega,\bm{a})=-\int_{D}\rho(z)\bm{G}^{(2)}(x,z,\omega)^2\bm{a}dz, \quad  x\in U.
\end{eqnarray}
For the difference $\bm{u}_1-\bm{u}_1^{(2)}$, we have the following estimate.

\begin{lemma}\label{lm:5.2}
Under the assumptions in Theorem \ref{thm3}, it holds for $x\in U$ that 
\begin{eqnarray*}%\label{d16}
|\bm{u}_1(x,\omega,\bm{a})-\bm{u}_1^{(2)}(x,\omega,\bm{a})|\leq
C\omega^{-3}\quad a.s.,
\end{eqnarray*}
where the constant $C$ depends on the distance between $U$ and $D$.
\end{lemma}

\begin{proof}
Using (\ref{d1}) and (\ref{d15}), for $x,y\in U$ and $z\in D$, we obtain
\begin{align*}
|\bm{u}_1(x,\omega,\bm{a})-\bm{u}_1^{(2)}(x,\omega,\bm{a})|
&\leq\left|\int_D\rho(z)(\bm{G}(x,z,\omega)-\bm{G}^{(2)}(x,z,\omega))
\bm{G}(x,z, \omega)\bm{a}dz\right|\\
&\quad +\left|\int_D\rho(z)\bm{G}^{(2)}(x,z,\omega)(\bm{G}(x,z,\omega)-
\bm{G}^{(2)}(x,z,\omega))\bm{a}dz\right|\\
&=:{\rm J}_1+{\rm J}_2.
\end{align*}
For ${\rm J}_1$, we have from \eqref{d10} that 
\begin{align*}
{\rm J}_1&\leq
\|\rho\|_{H_0^{-1}(D)}\|(\bm{G}(x,\cdot,\omega)-\bm{G}^{(2)}(x,\cdot,
\omega))\bm{G}(x,\cdot,\omega)\bm{a}\|_{\bm H^{1}(D)}\\
&\leq\|\rho\|_{H_0^{-1}(D)}\big[\|\nabla(\bm{G}(x,\cdot,\omega)-
\bm{G}^{(2)}(x, \cdot,\omega))\|_{(L^2(D))^{2\times
2}}\|\bm{G}(x,\cdot,\omega)\bm{a}\|_{\bm L^{\infty}(D)}\\
&\quad +\|\bm{G}(x,\cdot,\omega)-\bm{G}^{(2)}(x,\cdot,\omega)\|_{(L^2(D))^{2\times
2}}\left(\|\bm{G}(x,\cdot,\omega)\bm{a}\|_{\bm L^{\infty}(D)}+\|\nabla\bm{G}(x,\cdot,\omega)\bm{a}\|_{\bm L^{\infty}(D)}\right)\big]\\
&\lesssim\big[\omega^{-\frac{5}{2}}\omega^{-\frac{1}{2}}+\omega^{-\frac{7
}{2}}(\omega^{-\frac{1}{2}}+\omega^{\frac{1}{2}})\big]\Big(\int_{D}|x-z|^{-7}
dz\Big)^{\frac12}\sup_{z\in D}|x-z|^{-\frac12}\\
&\lesssim  \omega^{-3},
\end{align*}
where we use the facts $\rho\in W^{\frac{m-2}2-\epsilon,p}_0(D)\subset H_0^{-1}(D)$ by choosing a sufficiently small $\epsilon>0$ and that there is a positive distance between $U$ and $D$.
Similarly, we can prove that ${\rm J}_2\lesssim\omega^{-3}$, which completes the proof.
\end{proof}

Let $\bm{u}_1^{(2)}=\big(u_{1,1}^{(2)},u_{1,2}^{(2)}\big)^\top$, where
\begin{eqnarray*}
u_{1,k}^{(2)}=-\sum_{i,j=1}^2\int_D\rho(z)G_{ki}^{(2)}(x,z,\omega)G_{ij}^{(2)}(x,z,\omega)a_jdz,\quad k=1,2
\end{eqnarray*}
and $a_j$ is the component of the vector $\bm{a}$. It follows from a
straightforward calculation that 
\begin{eqnarray}\notag
&&\mathbb{E}\left[\bm{u}_1^{(2)}(x,\omega_1,\bm{a})\cdot\overline{\bm{u}_1^{(2)}(x,\omega_2,\bm{a})}\right]
=\sum_{k,i,j,\tilde{i},\tilde{j}=1}^2a_ja_{\tilde{j}}\times\\\label{d19}
&&\int_{\mathbb
R^2}\int_{\mathbb{R}^2}G_{ki}^{(2)}(x,z,\omega_1)G_{ij}^{(2)}(x,z,
\omega_1)\overline{G_{k\tilde{i}}^{(2)}(x,z',\omega_2)}\overline{G_{\tilde{i}\tilde{j}}^{(2)}(x,z',\omega_2)}\mathbb{E}
[\rho(z)\rho(z')]dzdz',
\end{eqnarray}
where the entries
$G_{kl}^{(2)}$ in $\bm G^{(2)}$ can be expressed by
\begin{align*}
G_{kl}^{(2)}(x,z,\omega)&=\frac{\rm i}4\sum_{j=0}^2\Bigg[\frac{b_j^{(0)}c_{\rm s}^{-j+\frac32}\delta_{kl}}{\omega^{j+\frac12}|x-z|^{j+\frac12}}+\frac{{\rm i}b_j^{(1)}c_{\rm s}^{-j+\frac12}\delta_{kl}}{\omega^{j+\frac32}|x-z|^{j+\frac32}}-\frac{b_j^{(2)}c_{\rm s}^{-j+\frac32}(x_k-z_k)(x_l-z_l)}{\omega^{j+\frac12}|x-z|^{j+\frac52}}\Bigg]e^{{\rm i}c_{\rm s}\omega|x-z|}\\
&\quad -\frac{\rm i}4\sum_{j=0}^2\bigg[\frac{{\rm i}b_j^{(1)}c_{\rm p}^{-j+\frac12}\delta_{kl}}{\omega^{j+\frac32}|x-z|^{j+\frac32}}-\frac{b_j^{(2)}c_{\rm p}^{-j+\frac32}(x_k-z_k)(x_l-z_l)}{\omega^{j+\frac12}|x-z|^{j+\frac52}}\bigg]e^{{\rm i}c_{\rm p}\omega|x-z|}
\end{align*}
and $\delta_{kl}$ is the Kronecker delta function which equals to 1 when $k=l$
and vanishes when $k\not= l$. Substituting the expression of
$G_{kl}^{(2)}$ into (\ref{d19}), we see that
$\mathbb{E}(\bm{u}_1^{(2)}(x,\omega_1,\bm{a})\cdot\overline{\bm{u}_1^{(2)}(x,
\omega_2,\bm{a})})$ is a linear combination of the following type of integral
\begin{eqnarray*}%\label{d21}
I_2(x,\omega_1,\omega_2):=\int_{\mathbb R^2}\int_{\mathbb R^2}e^{{\rm
i}(c_1\omega_1|x-z|-c_2\omega_2|x-z'|)}F_2(z,z',x)\mathbb{E}
[\rho(z)\rho(z')]dzdz',
\end{eqnarray*}
where $c_1,c_2\in\{2c_{\rm s},c_{\rm s}+c_{\rm p} ,2c_{\rm p}\}$ and 
\begin{eqnarray*}
F_2(z,z',x):=\frac{(x_1-z_1)^{d_{11}}(x_2-z_2)^{d_{12}}(x_1-z'_1)^{d_{21}}
(x_2-z'_2)^{d_{22}}}{|x-z|^{d_1}|x-z'|^{d_2}}. 
\end{eqnarray*}

To estimate the integral $I_2$ in $\mathbb R^2$, we may consider the following general integral in $\mathbb R^d$:
\begin{equation}\label{eq:Id}
I_d(x,\omega_1,\omega_2):=\int_{\mathbb R^d}\int_{\mathbb R^d}e^{{\rm
i}(c_1\omega_1|x-z|-c_2\omega_2|x-z'|)}F_d(z,z',x)\mathbb{E}[\rho(z)\rho(z')]dzdz',
\end{equation}
where $c_1,c_2\in\{2c_{\rm s},c_{\rm s}+c_{\rm p} ,2c_{\rm p}\}$ and 
\begin{eqnarray*}
F_d(z,z',x):=\frac{(x_1-z_1)^{d_{11}}\cdots(x_d-z_d)^{d_{1d}}(x_1-z'_1)^{d_{21}}\cdots(x_d-z'_d)^{d_{2d}}}{|x-z|^{d_1}|x-z'|^{d_2}}.
\end{eqnarray*}
The following estimate holds for $I_d$, whose proof is technical and is given in Appendix \ref{app:1} to avoid a possible distraction from the presentation of the main results. 

\begin{lemma}\label{lemmaadd1}
For $\omega_1,\omega_2\geq 1$, the following estimate holds uniformly for
$x\in U$:
\begin{eqnarray}\label{d23}
|I_d(x,\omega_1,\omega_2)|\leq
C_M(\omega_1+\omega_2)^{-m}(1+|\omega_1-\omega_2|)^{-M},
\end{eqnarray}
where $M\in\mathbb{N}$ is an arbitrary integer and the positive constant $C_M$
depends on $M$. Moreover, if $\omega_1=\omega_2=\omega$, then the following
identity holds:
\begin{eqnarray}\label{d25}
I_d(x,\omega,\omega)=R_d(x,\omega)\omega^{-m}+O(\omega^{
-(m+1)}),
\end{eqnarray}
where
\begin{eqnarray*}
R_d(x,\omega)=\frac{2^{m}}{(c_1+c_2)^{m}}\int_{\mathbb R^d}e^{{\rm
i}(c_1-c_2)\omega|x-\zeta|}\frac{(x_1-\zeta_1)^{d_{11}+d_{21}}\cdots(x_d-\zeta_d)^{d_{1d}+d_{2d}}}{|x-\zeta|^{d_1+d_2}}\phi(\zeta)d\zeta.
\end{eqnarray*}

\end{lemma}

\begin{corollary}\label{coro:u2}
For $\omega_1,\omega_2\ge1$, the following estimates hold  uniformly for $x\in
U$:
\begin{align}\label{d24}
|\mathbb{E}(\bm{u}_1^{(2)}(x,\omega_1,
\bm{a})\cdot\overline{\bm{u}_1^{(2)}(x,\omega_2,\bm{a})})|&\leq 
C_M(\omega_1\omega_2)^{-1}(\omega_1+\omega_2)^{-m}
(1+|\omega_1-\omega_2|)^{-M},\\\label{d100}
|\mathbb{E}(\bm{u}_1^{(2)}(x,\omega_1,\bm{a})\cdot\bm{u}_1^{(2)}(x,
\omega_2,\bm{a}))|&\leq
C_M(\omega_1\omega_2)^{-1}(\omega_1+\omega_2)^{-M}
(1+|\omega_1-\omega_2|)^{-m},
\end{align}
where $M\in\mathbb{N}$ is arbitrary and $C_M$ is a constant depending on
$M$.
\end{corollary}

\begin{proof}
Since $\mathbb{E}(\bm{u}_1^{(2)}(x,\omega_1,
\bm{a})\cdot\overline{\bm{u}_1^{(2)}(x,\omega_2,\bm{a})})$ is a linear
combination of $I_2$, where the  coefficient of the highest order is 
$(\omega_1\omega_2)^{-1}$, it follows from  Lemma \ref{lemmaadd1} that 
the estimate \eqref{d24} holds. A simple calculation shows that $\mathbb
E(\bm{u}_1^{(2)}(x,\omega_1,\bm{a})\cdot\bm{u}_1^{(2)}(x,\omega_2,\bm{a}))$
is a linear combination of the following type of integral
\begin{eqnarray*}%\label{d55}
\tilde I_2=\int_{\mathbb R^{2}}\int_{\mathbb R^2}e^{{\rm
i}(c_1\omega_1|x-z|+c_2\omega_2|x-z'|)}F_2(z,z',x)\mathbb{E}[\rho(z)\rho(z')]dzdz'
\end{eqnarray*}
with the coefficient $(\omega_1\omega_2)^{-1}$. Clearly, $\tilde I_2$
is analogous to $I_2$ except that $\omega_2$ in $I_2$ is
replaced by $-\omega_2$ in $\tilde I_2$. Following the same proof as that for
the estimate of $I_2$, we may show that 
\begin{eqnarray*}%\label{d56}
|\tilde I_2(x,\omega_1,\omega_2)|\leq
C_M(\omega_1+\omega_2)^{-M}(1+|\omega_1-\omega_2|)^{-m},
\end{eqnarray*}
which implies that \eqref{d100} holds and completes the proof. 
\end{proof}

\begin{proof}[Proof of Theorem \ref{thm3}] 
Rewriting $\bm{u}_1=\bm{u}_1^{(2)}+(\bm{u}_1-\bm{u}_1^{(2)})$, we only need
to show that
\begin{eqnarray}\label{d84}
&&\lim_{Q\to\infty}\frac{1}{Q-1}\int_1^Q\omega^{m+2}\sum_{j=1}^2|\bm{u}_1^{(2)}(x,\omega,\bm{a}_j)|^2d\omega=\frac{c_{\rm s}^{6-m}+c_{\rm p}^{6-m}}{2^{m+6}\pi^2}\int_{\mathbb
R^2}\frac{1}{|x-\zeta|^2}\phi(\zeta)d\zeta,\\\label{d85}
&&\lim_{Q\to\infty}\frac{1}{Q-1}\int_1^Q\omega^{m+2}|\bm{u}_1(x,\omega,\bm{a}
)-\bm{u}_1^{(2)}(x,\omega,\bm{a})|^2d\omega=0,\\\label{d86}
&&\lim_{Q\to\infty}\frac{2}{Q-1}\int_1^Q\omega^{m+2}
\Re\big[\overline{\bm{u}_1^{(2)}(x,\omega,\bm{a})}(\bm{u}_1(x,\omega,\bm{a
})-\bm{u}_1^{(2)}(x,\omega,\bm{a}))\big]d\omega=0.
\end{eqnarray}

For \eqref{d84}, it follows from a straightforward calculation that 
\begin{align*}
&\mathbb{E}|\bm u_1^{(2)}(x,\omega,\bm a)|^2=\left|\frac{\rm i-1}{4\sqrt{\pi}}\right|^4\omega^{-2}\sum_{j,\tilde j=1}^2a_ja_{\tilde j}\int_{\mathbb R^2}\int_{\mathbb R^2}\mathbb{E}[\rho(z)\rho(z')]\sum_{k,i,\tilde i=1}^2\\
&\bigg[c_{\rm
s}^{\frac32}\bigg(\frac{\delta_{ki}}{|x-z|^{\frac12}}-\frac{(x_k-z_k)(x_i-z_i)}{
|x-z|^{\frac52}}\bigg)e^{{\rm i}c_{\rm s}|x-z|}+c_{\rm
p}^{\frac32}\frac{(x_k-z_k)(x_i-z_i)}{|x-z|^{\frac52}}e^{{\rm i}c_{\rm
p}|x-z|}\bigg]\\
&\times\bigg[c_{\rm
s}^{\frac32}\bigg(\frac{\delta_{ij}}{|x-z|^{\frac12}}-\frac{(x_i-z_i)(x_j-z_j)}{
|x-z|^{\frac52}}\bigg)e^{{\rm i}c_{\rm s}|x-z|}+c_{\rm
p}^{\frac32}\frac{(x_i-z_i)(x_j-z_j)}{|x-z|^{\frac52}}e^{{\rm i}c_{\rm
p}|x-z|}\bigg]\\
&\times\bigg[c_{\rm s}^{\frac32}\bigg(\frac{\delta_{k\tilde
i}}{|x-z'|^{\frac12}}-\frac{(x_k-z_k')(x_{\tilde i}-z_{\tilde
i}')}{|x-z'|^{\frac52}}\bigg)e^{-{\rm i}c_{\rm s}|x-z'|}+c_{\rm
p}^{\frac32}\frac{(x_k-z_k')(x_{\tilde i}-z_{\tilde
i}')}{|x-z'|^{\frac52}}e^{-{\rm i}c_{\rm p}|x-z'|}\bigg]\\
&\times\bigg[c_{\rm s}^{\frac32}\bigg(\frac{\delta_{\tilde i\tilde
j}}{|x-z'|^{\frac12}}-\frac{(x_{\tilde i}-z_{\tilde i}')(x_{\tilde j}-z_{\tilde
j}')}{|x-z'|^{\frac52}}\bigg)e^{-{\rm i}c_{\rm s}|x-z'|}+c_{\rm
p}^{\frac32}\frac{(x_{\tilde i}-z_{\tilde i}')(x_{\tilde j}-z_{\tilde
j}')}{|x-z'|^{\frac52}}e^{-{\rm i}c_{\rm p}|x-z'|}\bigg]dzdz'\\
&+O(\omega^{-(m+3)}),
\end{align*}
which, together with  Lemma \ref{lemmaadd1}, gives  
\begin{eqnarray}\label{d79}
\mathbb{E}|\bm{u}_1^{(2)}(x,\omega,\bm{a})|^2=T_2(x,\omega,\bm{a})\omega^{
-(m+2)}+O(\omega^{-(m+3)}).
\end{eqnarray}
Here 
\begin{align*}
T_2(x,\omega,\bm{a}) &= 2^{-6-m}\pi^{-2}c_{\rm s}^{6-m}\int_{\mathbb
R^2}\bigg[\frac{1}{|x-\zeta|^2}-\sum_{i,j=1}^2\frac{(x_i-\zeta_i)(x_j-\zeta_j)}{
|x-\zeta|^4} a_ia_j\bigg]\phi(\zeta)d\zeta\\
&\quad + 2^{-6-m}\pi^{-2}c_{\rm p}^{6-m}\int_{\mathbb
R^2}\bigg[\sum_{i,j=1}^2\frac{(x_i-\zeta_i)(x_j-\zeta_j)}{|x-\zeta|^4}
a_ia_j\bigg]\phi(\zeta)d\zeta
\end{align*}
with $a_i$ being the $i$th component of the vector $\bm{a}$. Let
$\bm{a}_1=(a_{11}, a_{12})^\top$ and $\bm{a}_2=(a_{21}, a_{22})^\top$ be two
orthonormal vectors, i.e., there exists some angle $\alpha$ such that $\bm
a_1=(\cos\alpha,\sin\alpha)^\top$ and $\bm a_2=(-\sin\alpha,\cos\alpha)^\top$.
It then holds
$a_{11}^2+a_{21}^2=1$,  $a_{12}^2+a_{22}^2=1$ and $a_{11}a_{12}+a_{21}a_{22}=0$,
which lead to 
\begin{eqnarray*}
\sum_{j=1}^2T_2(x,\omega,\bm{a}_j)=\frac{c_{\rm s}^{6-m}+c_{\rm p}^{6-m}}{2^{m+6}\pi^2}\int_{\mathbb
R^2}\frac{1}{|x-\zeta|^2}\phi(\zeta)d\zeta.
\end{eqnarray*} 
It follows from (\ref{d79}) and the above equation that we have
\begin{eqnarray*}
\sum_{j=1}^2\mathbb{E}|\bm{u}_1^{(2)}(x,\omega,\bm{a}_j)|^2=\frac{c_{\rm s}^{6-m}+c_{\rm p}^{6-m}}{2^{m+6}\pi^2}\int_{\mathbb
R^2}\frac{1}{|x-\zeta|^2}\phi(\zeta)d\zeta
\omega^{-(m+2)}+O(\omega^{-(m+3)}),
\end{eqnarray*}
which gives
\begin{eqnarray*}%\label{d83}
\lim_{Q\to\infty}\frac{1}{Q-1}\int_1^Q\omega^{m+2}\sum_{j=1}^2\mathbb{E}|\bm{
u}_1^{(2)}(x,\omega,\bm{a}_j)|^2d\omega=\frac{c_{\rm s}^{6-m}+c_{\rm p}^{6-m}}{2^{m+6}\pi^2}\int_{\mathbb
R^2}\frac{1}{|x-\zeta|^2}\phi(\zeta)d\zeta.
\end{eqnarray*}

To prove (\ref{d84}), based on the above equation, it suffices to prove
that
\begin{eqnarray}\label{d87}
\lim_{Q\to\infty}\frac{1}{Q-1}\int_1^QY(x,\omega, \bm{a})d\omega=0,
\end{eqnarray}
where $Y(x,\omega, \bm{a})$ is defined by
\begin{align*}
Y(x,\omega,\bm{a})&=\omega^{m+2}\left[|\bm{u}_1^{(2)}(x,\omega,\bm{a}
)|^2-\mathbb{E}|\bm{u}_1^{(2)}(x,\omega,\bm{a})|^2\right]\\
&=\omega^{m+2}\left\{[
\Re\bm{u}_1^{(2)}(x,\omega,\bm{a})]^2-\mathbb{E}[
\Re\bm{u}_1^{(2)}(x,\omega,\bm{a})]^2\right\}\\
&\quad +\omega^{m+2}\left\{[
\Im\bm{u}_1^{(2)}(x,\omega,\bm{a})]^2-\mathbb{E}[
\Im\bm{u}_1^{(2)}(x,\omega,\bm{a})]^2\right\},
\end{align*}
where $\Re$ and $\Im$ denote the real and imaginary parts of a complex number, respectively. Note that
\begin{eqnarray*}%\label{d89}
\mathbb{E}(Y(x,\omega_1,\bm{a})Y(x,\omega_2,\bm{a}))={\rm F}_1+{\rm
F}_2+{\rm F}_3+{\rm F}_4,
\end{eqnarray*}
where
\begin{eqnarray*}%\label{d90}
{\rm
F}_1=\omega_1^{m+2}\omega_2^{m+2}\mathbb{E}\bigg\{\left[(
\Re\bm{u}_1^{(2)}(x,\omega_1,\bm{a}))^2-\mathbb{E}(
\Re\bm{u}_1^{(2)}(x,\omega_1,\bm{a}))^2\right]\\
\times \left[(\Re\bm{u}_1^{(2)}(x,\omega_2,\bm{a}))^2-\mathbb{E}(
\Re\bm{u}_1^{(2)}(x,\omega_2,\bm{a}))^2\right]\bigg\},
\end{eqnarray*}
\begin{eqnarray*}%\label{d91}
{\rm
F}_2=\omega_1^{m+2}\omega_2^{m+2}\mathbb{E}\bigg\{\left[(
\Re\bm{u}_1^{(2)}(x,\omega_1,\bm{a}))^2-\mathbb{E}(
\Re\bm{u}_1^{(2)}(x,\omega_1,\bm{a}))^2\right]\\
\times \left[(\Im\bm{u}_1^{(2)}(x,\omega_2,\bm{a}))^2-\mathbb{E}(
\Im\bm{u}_1^{(2)}(x,\omega_2,\bm{a}))^2\right]\bigg\},
\end{eqnarray*}
\begin{eqnarray*}%\label{d92}
{\rm
F}_3=\omega_1^{m+2}\omega_2^{m+2}\mathbb{E}\bigg\{\left[(
\Im\bm{u}_1^{(2)}(x,\omega_1,\bm{a}))^2-\mathbb{E}(
\Im\bm{u}_1^{(2)}(x,\omega_1,\bm{a}))^2\right]\\
\times \left[(\Re\bm{u}_1^{(2)}(x,\omega_2,\bm{a}))^2-\mathbb{E}(
\Re\bm{u}_1^{(2)}(x,\omega_2,\bm{a}))^2\right]\bigg\},
\end{eqnarray*}
\begin{eqnarray*}%\label{d93}
{\rm 
F}_4=\omega_1^{m+2}\omega_2^{m+2}\mathbb{E}\bigg\{\left[(
\Im\bm{u}_1^{(2)}(x,\omega_1,\bm{a}))^2-\mathbb{E}(
\Im\bm{u}_1^{(2)}(x,\omega_1,\bm{a}))^2\right]\\
\times \left[(\Im\bm{u}_1^{(2)}(x,\omega_2,\bm{a}))^2-\mathbb{E}(
\Im\bm{u}_1^{(2)}(x,\omega_2,\bm{a}))^2\right]\bigg\}.
\end{eqnarray*}
The expression of $\bm u_1^{(2)}$ given in (\ref{d15}) implies that both $
\Re\bm{u}_1^{(2)}$ and $\Im\bm{u}_1^{(2)}$ are centered Gaussian random
fields. Applying Lemma 2.3 in \cite{LHL} and Corollary \ref{coro:u2}, we obtain 
\begin{align*}
{\rm  F}_1&=2\omega_1^{m+2}\omega_2^{m+2}\left[\mathbb{E}(
\Re\bm{u}_1^{(2)}(x,\omega_1,\bm{a})
\Re\bm{u}_1^{(2)}(x,\omega_2,\bm{a}))\right]^2\\
&=\frac{1}{2}\omega_1^{m+2}\omega_2^{m+2}\bigg\{\mathbb{E}\Big[
\Re(\bm{u}_1^{(2)}(x,\omega_1,\bm{a})\bm{u}_1^{(2)}(x,\omega_2,\bm{a}))+
\Re(\bm{u}_1^{(2)}(x,\omega_1,\bm{a})\overline{\bm{u}_1^{(2)}(x,\omega_2,
\bm {a})})\Big]\bigg\}^2\\
&\lesssim
\bigg[\frac{\omega_1^{\frac{m+2}{2}}\omega_2^{\frac{m+2}{2}}}{
(\omega_1\omega_2)(\omega_1+\omega_2)^M(1+|\omega_1-\omega_2|)^m}+\frac{\omega_1^{\frac{m+2}{2}}
\omega_2^{\frac{m+2}{2}}}{(\omega_1\omega_2)(\omega_1+\omega_2)^{m}(1+|\omega_1-\omega_2|)^M}
\bigg]^2\\
&\lesssim
\left[(1+|\omega_1-\omega_2|)^{-m}+(1+|\omega_1-\omega_2|)^{-M}\right]^2.
\end{align*}

By the similar arguments, we can obtain the same estimate for ${\rm F}_2$,
${\rm F}_3$, and ${\rm F}_4$. Thus, an application of Lemma 2.4 in \cite{LHL}
implies 
\begin{eqnarray*}%\label{d95}
\lim_{Q\to\infty}\frac{1}{Q-1}\int_1^QY(x,\omega,\bm{a})d\omega=0,
\end{eqnarray*}
which shows that (\ref{d87}) holds and therefore (\ref{d84}) holds. 

For (\ref{d85}), by Lemma \ref{lm:5.2}, we obtain from the fact $m\le d=2$ that 
\begin{align*}%\label{d96}
&\frac{1}{Q-1}\int_1^Q\omega^{m+2}|\bm{u}_1(x,\omega,\bm{a})-\bm{u}_1^{(2)}
(x,\omega,\bm{a})|^2d\omega\\
&\lesssim
\frac{1}{Q-1}\int_1^Q\omega^{m-4}d\omega=\frac{1}{m-3}\frac{Q^{m-3}-1} {Q-1}\to
0\quad \text{as } Q\to \infty.
\end{align*}
Combining (\ref{d84})--(\ref{d85}) and
the H\"{o}lder inequality, we may easily verify (\ref{d86}) and complete
the proof. 
\end{proof}

\subsection{The analysis of $\bm u_2$}
\label{sec:5.2}
This subsection is devoted to analyzing the term $\bm{u}_2$ in the Born
approximation (\ref{c2}), which is given by
\begin{eqnarray*}%\label{e1}
\bm{u}_2(x,\omega,
\bm{a})=\int_D\int_D\bm{G}(x,z,\omega)\rho(z)\bm{G}(z,z',\omega)\rho(z')\bm{G}
(z',x,\omega)\bm{a}dz'dz,
\end{eqnarray*}
for $x\in U$. The purpose is to show that the contribution of
$\bm u_2$ can also be ignored, which is presented in the following theorem.

\begin{theorem}\label{thm4}
Under the assumptions in Theorem \ref{theorem1}, for all $x\in U$, it holds almost surely that 
\begin{eqnarray*}%\label{e2}
\lim_{Q\to\infty}\frac{1}{Q-1}\int_1^Q\omega^{m+2}|\bm{u}_2(x,\omega,
\bm{a})|^2d\omega=0.
\end{eqnarray*} 
\end{theorem}

To prove Theorem \ref{thm4}, motivated by \cite{LPS08} in the acoustic wave case, we decompose $\bm u_2$ into several terms by defining the following auxiliary functions: 
\begin{align}\label{e3}
\bm{u}_{2,l}(x,\omega,
\bm{a})&=\int_D\int_D\bm{G}^{(0)}(x,z,\omega)\rho(z)\bm{G}(z,z',
\omega)\rho(z')\bm{G}(z',x,\omega)\bm{a}dz'dz,\\\label{e4}
\bm{u}_{2,r}(x,\omega,
\bm{a})&=\int_D\int_D\bm{G}^{(0)}(x,z,\omega)\rho(z)\bm{G}(z,z',
\omega)\rho(z')\bm{G}^{(0)}(z',x,\omega)\bm{a}dz'dz,\\\label{eadd1}
\bm{v}(x,\omega,
\bm{a})&=\int_D\int_D\bm{G}^{(0)}(x,z,\omega)\rho(z)\bm{G}^{(0)}(z,z',
\omega)\rho(z')\bm{G}^{(0)}(z',x,\omega)\bm{a}dz'dz,
\end{align}
where $\bm G^{(0)}$ is defined in \eqref{d8}. It is clear to note that
$\bm u_2=(\bm u_2-\bm u_{2,l})+(\bm u_{2,l}-\bm u_{2,r})+(\bm u_{2,r}-\bm v)+\bm v$. 
To estimate these terms, the following preliminary results on $\bm{G}$, $\bm{G}^{(0)}$, and their difference are needed.

\begin{lemma}\label{lm:G0}
Let $s\in[0,1]$ and $U\subset\mathbb R^2\backslash\overline{D}$ be
a bounded and convex domain with a positive distance
from $D$. 
\begin{itemize}
\item[(i)] For any $p\in(1,\infty)$ and $x\in U$, it holds
\begin{align*}
\|\bm G(x,\cdot,\omega)\|_{(W^{s,p}(D))^{2\times2}}&\lesssim\omega^{-\frac12+s},\\
\|\bm G^{(0)}(x,\cdot,\omega)\|_{(W^{s,p}(D))^{2\times2}}&\lesssim\omega^{-\frac12+s},\\
\|\bm G(x,\cdot,\omega)-\bm G^{(0)}(x,\cdot,\omega)\|_{(W^{s,p}(D))^{2\times2}}&\lesssim\omega^{-\frac32+s}.
\end{align*}
\item[(ii)] For any $p\in(1,\frac43)$, it holds
\[
\|\bm G(\cdot,\cdot,\omega)-\bm G^{(0)}(\cdot,\cdot,\omega)\|_{(W^{s,p}(D\times D))^{2\times2}}\lesssim\omega^{-\frac32+s}.
\]
\end{itemize}
\end{lemma}

\begin{proof}
Results in (i) can be easily obtained by \eqref{d10} and the interpolation between the spaces $L^p(D)$ and $W^{1,p}(D)$. Next is to show (ii).

According to \eqref{d10}, we get
\begin{align*}
\|\bm G(\cdot,\cdot,\omega)-\bm G^{(0)}(\cdot,\cdot,\omega)\|_{(L^{p}(D\times D))^{2\times2}}\lesssim\omega^{-\frac32}\left(\int_D\int_D|z-z'|^{-\frac32p}dzdz'\right)^{\frac1p}\lesssim\omega^{-\frac32}, 
\end{align*}
where we use the facts that there exists a constant $R>0$ such that $|z-z'|<R$ for any $z,z'\in D$, and 
\[
\int_D\int_D|z-z'|^{-\frac32p}dzdz'\lesssim\int_0^Rr^{-\frac32p+1}dr<\infty
\]
for any $p\in(1,\frac43)$. Similarly,
\[
\|\bm G(\cdot,\cdot,\omega)-\bm G^{(0)}(\cdot,\cdot,\omega)\|_{(W^{1,p}(D\times D))^{2\times2}}\lesssim\omega^{-\frac12}.
\]
Finally, the result in (ii) is obtained by the interpolation.
\end{proof}

The operator $K_\omega$, defined by \eqref{eq:K}, satisfies the following estimates when restricted to bounded domains, where the proof is given for a more general case $m\in(1,2]$. We also refer to \cite[Lemma 5]{LPS08} for the acoustic wave case with $m=2$. 

\begin{lemma}\label{lm:K2}
Let $\rho$ satisfy Assumption \ref{as:rho}. For any $s\in(\frac{2-m}2,\frac12)$, $q\in(1,\infty)$ and $\omega\ge1$, it holds that
\begin{eqnarray*}
\|K_{\omega}\|_{\mathcal{L}( \bm W^{s,2q}(D))}\lesssim
\omega^{-1+2s}.
\end{eqnarray*}
\end{lemma}

\begin{proof}
For any $\bm f,\bm g\in\bm C_0^\infty(D)$, denote by $\tilde{\bm f},\tilde{\bm g}$ the zero extensions of $\bm f,\bm g$ in $\mathbb R^2$ such that $\tilde{\bm f},\tilde{\bm g}\in\bm C_0^\infty$.  Using Lemma \ref{lm:H} leads to
\begin{align*}
|\langle H_\omega\bm f,\bm g\rangle|=|\langle H_\omega\tilde{\bm f},\tilde{\bm g}\rangle|\lesssim\omega^{-1+2s+s_1+s_2}\|\tilde{\bm f}\|_{\bm H_1^{-s-s_1}}\|\tilde{\bm g}\|_{\bm H_1^{-s-s_2}},
\end{align*}
where
\begin{align*}
\|\tilde{\bm f}\|_{\bm H_1^{-s-s_1}}\lesssim\|\mathcal{J}^{-s}\bm f\|_{\bm H^{-s_1}(D)}\lesssim\|\mathcal{J}^{-s}\bm f\|_{\bm L^{\tilde p}(D)}\lesssim\|\bm f\|_{\bm W^{-s,\tilde p}(D)},\\
\|\tilde{\bm g}\|_{\bm H_1^{-s-s_2}}\lesssim\|\mathcal{J}^{-s}\bm g\|_{\bm H^{-s_2}(D)}\lesssim\|\mathcal{J}^{-s}\bm g\|_{\bm L^{\tilde q'}(D)}\lesssim\|\bm g\|_{\bm W^{-s,\tilde q'}(D)}
\end{align*}
according to the Sobolev embeddings
\begin{align*}
\bm L^{\tilde p}(D)\hookrightarrow\bm H^{-s_1}(D)\quad\text{for}\quad s_1\ge\frac 2{\tilde p}-1,\\
\bm L^{\tilde q'}(D)\hookrightarrow\bm H^{-s_2}(D)\quad\text{for}\quad s_2\ge1-\frac 2{\tilde q}
\end{align*}
with $1<\tilde p\le2\le\tilde q<\infty$ and $\tilde q'$ satisfying $\frac1{\tilde q}+\frac1{\tilde q'}=1$ (cf. \cite[Theorem 3.1]{LP}). By choosing $s_1=\frac 2{\tilde p}-1$ and $s_2=1-\frac 2{\tilde q}$, we have
\[
|\langle H_\omega\bm f,\bm g\rangle|\lesssim\omega^{-1+2s+2(\frac1{\tilde p}-\frac1{\tilde q})}\|\bm f\|_{\bm W^{-s,\tilde p}(D)}\|\bm g\|_{\bm W^{-s,\tilde q'}(D)},
\]
and hence
\[
\|H_\omega\|_{\mathcal{L}(\bm W^{-s,\tilde p}(D),\bm W^{s,\tilde q}(D))}\lesssim\omega^{-1+2s+2(\frac1{\tilde p}-\frac1{\tilde q})}.
\]
Note that $K_\omega\bm u=H_\omega(\rho\bm u)$. For any $\bm u\in\bm C_0^\infty(D)\subset\bm W^{s,2q}(D)$, we obtain from  \cite[Lemma 2]{LPS08} and Lemma \ref{le1} that 
\begin{align*}
\|K_\omega\bm u\|_{\bm W^{s,2q}(D)}&\lesssim\|H_\omega\|_{\mathcal{L}(\bm W^{-s,(2q)'}(D),\bm W^{s,2q}(D))}\|\rho\bm u\|_{\bm W^{-s,(2q)'}(D)}\\
&\lesssim\|H_\omega\|_{\mathcal{L}(\bm W^{-s,(2q)'}(D),\bm W^{s,2q}(D))}\|\rho\|_{W^{-s,p}(D)}\|\bm u\|_{\bm W^{s,2q}(D)}\\
&\lesssim\omega^{-1+2s+2(1-\frac1q)}\|\bm u\|_{\bm W^{s,2q}(D)},
\end{align*}
where $p$ satisfies $\frac1p+\frac1q=1$ and  $(2q)'$ satisfies $\frac1{2q}+\frac1{(2q)'}=1$. The proof is then completed based on the fact that $\bm C_0^\infty(D)$ is dense in $\bm L^{2q}(D)$ (cf. \cite[Section 2.30]{AF03}).
\end{proof}
According to Lemmas \ref{lm:u0}, \ref{lm:G0} and \ref{lm:K2}, we get for any $x\in U$ and $\omega\ge1$ that
\begin{align*}
&|\bm{u}_2(x,\omega, \bm{a})-\bm{u}_{2,l}(x,\omega, \bm{a})|\\
&=\left|\Big\langle \rho,
\big[\bm G(x,\cdot,\omega)-\bm G^{(0)}(x,\cdot,\omega)\big]\int_D\bm G(\cdot,z',\omega)\rho(z')\bm G(z',x,\omega)\bm adz'\Big\rangle\right|\\
&\le\|\rho\|_{W^{-s,p}_0(D)}\left\|\left[\bm G(x,\cdot,\omega)-\bm G^{(0)}(x,\cdot,\omega)\right]K_{\omega}\bm u_0(\cdot,x)\right\|_{\bm W^{s,q}(D)}\\
&\le\|\rho\|_{W^{-s,p}_0(D)}\|\bm G(x,\cdot,\omega)-\bm G^{(0)}(x,\cdot,\omega)\|_{(W^{s,2q}(D))^{2\times2}}\|K_{\omega}\|_{\mathcal{L}(\bm W^{s,2q}(D))}\|\bm u_0(\cdot,x)\|_{\bm W^{s,2q}(D)}\\
&\lesssim\omega^{-\frac32+s-1+2s+2(1-\frac1q)-\frac12+s}=\omega^{-3+4s+2(1-\frac1q)}
\end{align*} 
for any $s\in(\frac{2-m}2,\frac12)$, $q\in(1,\infty)$ and $p$ satisfying $\frac1p+\frac1q=1$. Taking $q=1+s$, we then deduce
\begin{align}\label{eq:u2l}
\lim_{Q\to\infty}\frac{1}{Q-1}\int_1^Q\omega^{m+2}|\bm{u}_2(x,\omega,
\bm{a})-\bm{u}_{2,l}(x,\omega,
\bm{a})|^2d\omega
\lesssim\lim_{Q\to\infty}\frac{Q^{m-3+8s+\frac{4s}{1+s}}-1}{Q-1}=0
\end{align}
almost surely by choosing $s\in(\frac{2-m}2,\frac12)$ such that $m-3+8s+\frac{4s}{1+s}<m-3+12s<1$. Note that such an $s$ can be chosen in the interval $(\frac{2-m}2,\frac{4-m}{12})$, which is not empty due to the fact $m\in(\frac53,2]$ under assumptions in Theorem \ref{theorem1}.

Similarly, for the term $\bm u_{2,l}-\bm u_{2,r}$, we get
\begin{align*}
&|\bm{u}_{2,l}(x,\omega, \bm{a})-\bm{u}_{2,r}(x,\omega, \bm{a})|\\
&=\left|\Big\langle \rho,
\bm G^{(0)}(x,\cdot,\omega)\int_D\bm G(\cdot,z',\omega)\rho(z')\big[\bm G(z',x,\omega)-\bm G^{(0)}(z',x,\omega)\big]\bm adz'\Big\rangle\right|\\
&\le\|\rho\|_{W^{-s,p}_0(D)}\left\|\bm G^{(0)}(x,\cdot,\omega)K_{\omega}\big((\bm G-\bm G^{(0)})\bm a\big)(\cdot,x)\right\|_{\bm W^{s,q}(D)}\\
&\le\|\rho\|_{W^{-s,p}_0(D)}\|\bm G^{(0)}(x,\cdot,\omega)\|_{(W^{s,2q}(D))^{2\times2}}\|K_{\omega}\|_{\mathcal{L}(\bm W^{s,2q}(D))}\|(\bm G-\bm G^{(0)})(\cdot,x,\omega)\bm a\|_{\bm W^{s,2q}(D)}\\
&\lesssim\omega^{-\frac12+s-1+2s+2(1-\frac1q)-\frac32+s}
=\omega^{-3+4s+2(1-\frac1q)}
\end{align*} 
for any $s\in(\frac{2-m}2,\frac12)$, $q\in(1,\infty)$ and $p$ satisfying $\frac1p+\frac1q=1$, which leads to 
\begin{align}\label{eq:u2r}
\lim_{Q\to\infty}\frac{1}{Q-1}\int_1^Q\omega^{m+2}|\bm{u}_{2,l}(x,\omega,
\bm{a})-\bm{u}_{2,r}(x,\omega,
\bm{a})|^2d\omega=0
\end{align}
in the almost surely sense by taking $q=1+s$.

The term $\bm u_{2,r}-\bm v$ satisfies
\begin{align}\label{eq:u-v}
&\left|\bm{u}_{2,r}(x,\omega,\bm{a})-\bm{v}(x,\omega,\bm{a})\right|\notag\\
&=\left|\int_D\int_D\bm G^{(0)}(x,z,\omega)\rho(z)\big[\bm G(z,z',\omega)-\bm G^{(0)}(z,z',\omega)\big]\rho(z')\bm G^{(0)}(z',x,\omega)\bm adz'dz\right|\notag\\
&\lesssim\|\bm G(\cdot,\cdot,\omega)-\bm G^{(0)}(\cdot,\cdot,\omega)\|_{(W^{2s,p}(D\times D))^{2\times2}}\|\bm f\|_{\bm W_0^{-2s,q}(D\times D)},
\end{align}
where $p\in(1,\frac43)$, $q$ satisfies $\frac1p+\frac1q=1$, and
\[
\bm f(z,z'):=\rho(z)\rho(z')\bm G^{(0)}(x,z,\omega)\bm G^{(0)}(z',x,\omega)\bm a.
\]
Note that for any $\bm g\in \bm W^{2s,p}(D\times D)$ and $s\in(\frac{2-m}2,\frac12)$, 
\begin{align*}
|\langle\bm f,\bm g\rangle|&\lesssim\|\rho\otimes\rho\|_{W_0^{-2s,q}(D\times D)}\|(\bm G^{(0)}(x,\cdot,\omega)\bm G^{(0)}(\cdot,x,\omega)\bm a)\cdot\bm g\|_{W^{2s,p}(D\times D)}\\
&\lesssim\|\rho\|_{W^{-s,\infty}(D)}^2\|\bm G^{(0)}(x,\cdot,\omega)\bm G^{(0)}(\cdot,x,\omega)\bm a\|_{\bm W^{2s,\infty}(D\times D)}\|\bm g\|_{\bm W^{2s,p}(D\times D)}\\
&\lesssim\omega^{-1+2s}\|\bm g\|_{\bm W^{2s,p}(D\times D)}
\end{align*}
according to \eqref{d10}, Lemma \ref{le1}, and the fact that $\rho\otimes\rho\in W^{-2s,\infty}_0(D\times D)\subset W^{-2s,q}_0(D\times D)$ for any $\rho\in W_0^{-s,\infty}(D)$ (cf. \cite{LPS08}). 
As a result, 
\[
\|\bm f\|_{\bm W^{-2s,q}_0(D\times D)}\lesssim\omega^{-1+2s}.
\]
Hence, \eqref{eq:u-v} turns to be
\[
\left|\bm{u}_{2,r}(x,\omega,\bm{a})-\bm{v}(x,\omega,\bm{a})\right|\lesssim\omega^{-\frac52+4s},
\]
which leads to 
\begin{align}\label{eq:u2r-v}
\lim_{Q\to\infty}\frac{1}{Q-1}\int_1^Q\omega^{m+2}|\bm{u}_{2,r}(x,\omega,
\bm{a})-\bm v(x,\omega,
\bm{a})|^2d\omega\lesssim\lim_{Q\to\infty}\frac{Q^{m-2+8s}-1}{Q-1}=0
\end{align}
almost surely by choosing $s\in(\frac{2-m}2,\frac12)$ with $m\in(\frac53,2]$ such that $m-2+8s<1$.

Finally, the result in Theorem \ref{thm4} is obtained by combining \eqref{eq:u2l}, \eqref{eq:u2r}, \eqref{eq:u2r-v} and the following lemma, whose proof is rather technical and is given in Appendix \ref{app:2}. 

\begin{lemma}\label{lm:v}
Under assumptions in Theorem \ref{theorem1}, for all $x\in U$, the auxiliary function $\bm v$ defined in \eqref{eadd1} satisfies 
\begin{align}\label{eq:v}
\lim_{Q\to\infty}\frac{1}{Q-1}\int_1^Q\omega^{m+2}|\bm v(x,\omega,
\bm{a})|^2d\omega=0
\end{align} 
in the almost surely sense.
\end{lemma}

\subsection{The analysis of $\bm b$}\label{sec:5.3}

For any $x,y\in U$ and $s\in(\frac{2-m}2,\frac12)$,  it follows from Lemmas \ref{lm:K}, \ref{lm:u0} and \eqref{d10} that
\begin{align*}
|\bm b(x,y)|&=\left|\sum_{j=3}^\infty\bm u_j(x,y)\right|=\left|\sum_{j=3}^\infty\langle\rho,\bm G(x,\cdot,\omega)\bm u_{j-1}(\cdot,y)\rangle\right|\\
&\lesssim\sum_{j=2}^\infty\|\rho\|_{W^{-s,\infty}(D)}\|\bm G(x,\cdot,\omega)\bm u_j(\cdot,y)\|_{\bm W^{s,1}(D)}\\
&\lesssim\sum_{j=2}^\infty\|\bm G(x,\cdot,\omega)\|_{(W^{s,\infty}(D))^{2\times 2}}\|\bm u_j(\cdot,y)\|_{\bm H^s(D)}\\
&\lesssim\sum_{j=2}^\infty\|\bm G(x,\cdot,\omega)\|_{(W^{s,\infty}(D))^{2\times 2}}\|K_\omega\|^j_{\mathcal{L}(\bm H^s_{-1})}\|\bm u_0(\cdot,y)\|_{\bm H^s_{-1}}\\
&\lesssim\sum_{j=2}^\infty\omega^{-\frac12+s}\left(\omega^{-1+2s}\right)^j\omega^{-\frac12+s}\lesssim\omega^{-3+6s}.
\end{align*}
Then it holds for $\bm b(x,\omega,\bm a_j):=\bm b(x,x)$ that
\begin{align*}
\frac{1}{Q-1}\int_1^Q\omega^{m+2}|\bm{b}(x,\omega,\bm{
a}_j)|^2d\omega &\lesssim\frac{1}{Q-1}\int_1^Q\omega^{
m+2+(-3+6s)2}d\omega\\
&\lesssim\frac{Q^{m-3+12s}-1}{Q-1}\to0\quad\text{as }
Q\to\infty
\end{align*}
by choosing $s\in(\frac{2-m}2,\frac12)$ and $s<\frac{4-m}{12}$. 
We mention that such an $s$ exists since $m>\frac53$ under assumptions in Theorem \ref{theorem1}.

\subsection{The proof of Theorem \ref{theorem1}}
\label{sec:5.4}

Based on the analysis of $\bm{u}_1$, $\bm{u}_2$ and $\bm{b}$, we are now able to prove the
main result: the strength $\phi$ in the principal symbol of the covariance
operator $Q_\rho$ can be uniquely determined by the amplitude of two scattered
fields averaged over the frequency band with
probability one. Here, the two scattered fields are associated with the incident
waves given by
$\bm G(x,y)\bm{a}_1$ and $\bm G(x,y)\bm{a}_2$ for any two orthonormal
vectors $\bm{a}_1$ and $\bm{a}_2$. 

Recall that the scattered field $\bm{u}^{ s}$ can be written as 
\begin{eqnarray*}%\label{f2}
\bm{u}^{s}(x,\omega,\bm{a}_j)=\bm{u}_1(x,\omega,\bm{a}_j)+\bm{u}_2(x,\omega,\bm{a}_j)+
\bm{b}(x,\omega,\bm{a}_j),\quad j=1,2,
\end{eqnarray*} 
where $\bm u_1$, $\bm u_2$ and $\bm b$ satisfy 
for $x\in U$ that 
\begin{eqnarray*}
&&\lim_{Q\to\infty}\frac{1}{Q-1}\int_1^Q\omega^{m+2}\sum_{j=1}^2|\bm{u}_1
(x,\omega,\bm{a}_j)|^2d\omega=\frac{c_{\rm s}^{6-m}+c_{\rm p}^{6-m}}{2^{m+6}\pi^2}\int_{\mathbb
R^2}\frac{1}{|x-\zeta|^{2}}\phi(\zeta)d\zeta,\\
&&\lim_{Q\to\infty}\frac{1}{Q-1}\int_1^Q\omega^{m+2}|\bm{u}_2(x,\omega,\bm{
a}_j)|^2d\omega=0,\\
&&\lim_{Q\to\infty}\frac{1}{Q-1}\int_1^Q\omega^{m+2}|\bm{b}(x,\omega,\bm{
a}_j)|^2d\omega=0.
\end{eqnarray*}
Using the H\"{o}lder inequality gives 
\begin{align*}
&\left|\frac{1}{Q-1}\int_1^Q\omega^{m+2}
\Re\left[\bm{u}_i(x,\omega,\bm{a}_j)\overline{\bm{b}(x,\omega,\bm{a}_j)}
\right]d\omega\right|\\
&\lesssim
\frac{1}{Q-1}\int_1^Q\omega^{m+2}|\bm{u}_i(x,\omega,\bm{a}_j)\|\bm{b}
(x,\omega,\bm{a}_j)|d\omega\\
&\lesssim
\left[\frac{1}{Q-1}\int_1^Q\omega^{m+2}|\bm{u}_i(x,\omega,\bm{a}_j
)|^2\right]^{\frac{1}{2}}\left[\frac{1}{Q-1}\int_1^Q\omega^{m+2}|\bm{b}
(x,\omega,\bm{a}_j)|^2\right]^{\frac{1}{2}}\to0
\end{align*}
as $Q\to\infty$ for $i,j=1,2$ and similarly
\[
\left|\frac{1}{Q-1}\int_1^Q\omega^{m+2}
\Re\left[\bm{u}_1(x,\omega,\bm{a}_j)\overline{\bm{u}_2(x,\omega,\bm{a}_j)}
\right]d\omega\right|\to0
\]
as $Q\to\infty$ for $j=1,2$. 
Hence, we obtain
\begin{align*}
&\lim_{Q\to\infty}\frac{1}{Q-1}\int_1^Q\omega^{m+2}\sum_{j=1}^2|\bm{u}^
{ s}(x,\omega,\bm{a}_j)|^2d\omega\\
&=\lim_{Q\to\infty}\frac{1}{Q-1}\int_1^Q\omega^{m+2}\sum_{j=1}
^2\left|\bm{u}_1(x,\omega,\bm{a}_j)+\bm{u}_2(x,\omega,\bm{a}_j)+\bm{b}(x,
\omega,\bm{a}_j)\right|^2d\omega\\
&=\lim_{Q\to\infty}\frac{1}{Q-1}\int_1^Q\omega^{m+2}\sum_{j=1}^2
|\bm{u}_1(x,\omega,\bm{a}_j)|^2d\omega\\
&=\frac{c_{\rm s}^{6-m}+c_{\rm p}^{6-m}}{2^{m+6}\pi^2}\int_{\mathbb
R^2}\frac{1}{|x-\zeta|^{2}}\phi(\zeta)d\zeta.
\end{align*}  

It follows from \cite[Lemma 3.8]{LHL} that the function $\phi$ can be
uniquely determined from the integral equation (\ref{a7}) for all $x\in U$, which completes the proof of Theorem \ref{theorem1}.

\section{Conclusion}
\label{sec:6}

We have studied the direct and inverse scattering problems for the time-harmonic elastic wave
equation with a random potential in two dimensions. The potential is assumed
to be a microlocally isotropic generalized Gaussian random field whose
covariance is a classical pseudo-differential operator. For such a distribution
potential, we deduce the equivalence between the direct scattering problem and
the Lippmann--Schwinger integral equation which is shown to have a unique
solution. Employing the Born approximation in the high frequency regime and
microlocal analysis for the Fourier integral operators, we establish the
connection between the principal symbol of the covariance operator for the
random potential and the amplitude of the scattered field generated by a single
realization of the random potential. Based on the identity, we obtain the
uniqueness for the recovery of the micro-correlation strength of the random
potential.

For the three-dimensional case, the well-posedness of the direct scattering problem can be obtained based on the same procedure as the two-dimensional case. The convergence of the Born series and the estimate for $\bm u_1$ obtained in Theorem \ref{thm3} can also be extended to the three-dimensional case. However, what is different from the two-dimensional case is that the Green tensor in three dimensions does not decay with respect to the frequency $\omega$. It is unclear whether the contribution of higher order terms can be neglected in three dimensions. Hence, a frequency-dependence
assumption of the potential $\rho$, e.g., $\rho(x,\omega)=\rho(x)\omega^{-\theta}$ with $\theta>\frac{m-1}{2}$, might be required to uniquely recover the micro-correlation strength by using near-field data \cite{HLP}.
A possible way to avoid this difficulty in the three-dimensional case is to recover the strength by using the far-field patterns as the data \cite{LLW}. Another interesting and challenging problem is to investigate the case
where both the source and potential are random. We will report the progress on these problems elsewhere in the future.

\section{Acknowledgements}

The research of P. Li is supported in part by the NSF grant DMS-1912704.

\appendix
\section{Proof of Lemma \ref{lemmaadd1}}\label{app:1}

For $x\in U$, it holds
\begin{align*}
I_d(x,\omega_1,\omega_2)&=\int_{\mathbb R^d}\int_{\mathbb R^d}e^{{\rm
i}(c_1\omega_1|x-z|-c_2\omega_2|x-z'|)}F_d(z,z',x)\mathbb{E}(\rho(z)\rho(z'))dzdz'\\
&=\int_{\mathbb R^d}\int_{\mathbb R^d}e^{{\rm
i}(c_1\omega_1|x-z|-c_2\omega_2|x-z'|)}F_d(z,z',x)B_1(z,z',x)dzdz',
\end{align*}
where $B_1(z,z',x)=\mathcal{K}_{\rho}(z,z')\vartheta(x)$ with $\mathcal{K}_{\rho}$ being the kernel function of $\rho$ and $\vartheta(x)\in C_0^\infty$ such that $\vartheta|_U\equiv1$ and supp$(\vartheta)\subset\mathbb{R}^d\backslash\overline{D}$. 
Since $\rho$ is an isotropic Gaussian random field of order $-m$, we
can represent $\mathcal{K}_{\rho}$ in terms of its symbol by
\begin{eqnarray*}\label{d27}
\mathcal{K}_{\rho}(z,z')=\frac1{(2\pi)^{d}}\int_{\mathbb R^d}e^{{\rm
i}(z-z')\cdot\xi}\sigma(z,\xi)d\xi,
\end{eqnarray*}
where $\sigma\in S_{1,0}^{-m}(\mathbb R^d\times\mathbb R^d)$ is
the symbol of the covariance operator $Q_{\rho}$ with $S_{1,0}^{-m}(\mathbb R^d\times\mathbb R^d)$ being the space of symbols of order $-m$ which is defined by
\begin{eqnarray*}
S_{1,0}^{-m}(\mathbb R^d\times\mathbb R^d):=\Big\{a(x,\xi)\in C^{\infty}(\mathbb
R^d\times\mathbb R^d):|\partial_{\xi}^{\alpha}\partial_{x}^{\beta}|\leq
C_{\alpha,\beta}(1+|\xi|)^{-m-|\alpha|}\Big\}.
\end{eqnarray*} 
Here $\alpha$, $\beta$ are multiple indices with
$|\alpha|:=\sum_{j=1}^d\alpha_j$ for $\alpha = (\alpha_1,...,\alpha_d)^\top$.
Therefore
\begin{eqnarray*}%\label{d29}
B_1(z,z',x)=\frac1{(2\pi)^{d}}\int_{\mathbb
R^d}e^{{\rm i}(z-z')\cdot\xi}\sigma_1(z,x,\xi)d\xi,
\end{eqnarray*}
where $\sigma_1(z,x,\xi)=\sigma(z,\xi)\vartheta(x)\in S_{1,0}^{-m}$ has the
principal symbol
$\sigma_1^{p}(z,x,\xi)=\phi(z)|\xi|^{-m}\vartheta(x)$.
Moreover, $B_1$ is a conormal distribution in $\mathbb R^{3d}$ of the
H\"ormander type and is compactly supported in $D^\vartheta:=D\times
D\times\text{supp}(\vartheta)$, which has conormal singularity on the surface
$S:=\{(z,z',x)\in\mathbb{R}^{3d}:z-z'=0\}$, and is invariant under the change of
coordinates (cf. \cite{L3}).

To estimate the integral $I_d$, we apply the coordinate transformations
$\tau,\eta$ and $\gamma$ which are introduced in \cite{LW}. 

Define the invertible transformation $\tau:\mathbb R^{3d}\to\mathbb R^{3d}$ by $\tau(z,z',x)=(g,h,x)$,
where $g=(g_1,\cdots,g_d)$ and $h=(h_1,\cdots,h_d)$ with
\begin{align*}
g_1=\frac{1}{2}(|x-z|-|x-z'|),\quad
g_2=\frac12\bigg[|x-z|\arcsin\Big(\frac{z_1-x_1}{|x-z|}
\Big)-|x-z'|\arcsin\Big(\frac{z'_1-x_1}{|x-z'|}\Big)\bigg],\\
h_1=\frac{1}{2}(|x-z|+|x-z'|),\quad
h_2=\frac12\bigg[|x-z|\arcsin\Big(\frac{z_1-x_1}{|x-z|}
\Big)+|x-z'|\arcsin\Big(\frac{z'_1-x_1}{|x-z'|}\Big)\bigg]
\end{align*}
if $d=2$, and 
\begin{eqnarray*}
&& g_1=\frac12\left(|x-z|-|x-z'|\right),\quad
h_1=\frac12\left(|x-z|+|x-z'|\right),\\
&&g_2=\frac12\bigg[|x-z|\arccos\Big(\frac{z_3-x_3}{|x-z|}
\Big)-|x-z'|\arccos\Big(\frac{z_3'-x_3}{|x-z'|}\Big)\bigg],\\
&&h_2=\frac12\bigg[|x-z|\arccos\Big(\frac{z_3-x_3}{|x-z|}
\Big)+|x-z'|\arccos\Big(\frac{z_3'-x_3}{|x-z'|}\Big)\bigg],\\
&&g_3=\frac12\bigg[|x-z|\arctan\Big(\frac{z_2-x_2}{z_1-x_1}
\Big)-|x-z'|\arctan\Big(\frac{z_2'-x_2}{z_1'-x_1}\Big)\bigg],\\
&&h_3=\frac12\bigg[|x-z|\arctan\Big(\frac{z_2-x_2}{z_1-x_1}
\Big)+|x-z'|\arctan\Big(\frac{z_2'-x_2}{z_1'-x_1}\Big)\bigg].
\end{eqnarray*}
if $d=3$. 
We then get
\begin{align*}
I_d(x,\omega_1,\omega_2)&=\int_{\mathbb R^d}\int_{\mathbb R^d}e^{{\rm
i}\left((c_1\omega_1+c_2\omega_2)\frac{|x-z|-|x-z'|}2+(c_1\omega_1-c_2\omega_2)\frac{|x-z|+|x-z'|}2\right)}F_d(z,z',x)B_1(z,z',x)dzdz'\\
&=\int_{\mathbb R^d}\int_{\mathbb R^d}e^{{\rm
i}\left((c_1\omega_1+c_2\omega_2)g\cdot e_1+(c_1\omega_1-c_2\omega_2)h\cdot e_1\right)}B_2(g,h,x)dgdh,
\end{align*}
where $e_1=(1,0,\cdots,0)\in\mathbb R^d$ and 
\begin{align*}
B_2(g,h,x)&=B_1(\tau^{-1}(g,h,x))\left[F_d(\tau^{-1}(g,h,x))|\det((\tau^{-1})'(g,h,x))|\right]\\
&=:B_1(\tau^{-1}(g,h,x))L^\tau(g,h,x).
\end{align*}

The way to get the detailed expression of $B_2$ is exactly the same as the
procedure used in \cite{LW}, which is based on the transformations $\eta$
defined by $\eta(z,z',x)=(v,w,x)$ with $v=z-z'$ and $w=z+z'$, and
$\gamma:=\eta\circ\tau^{-1}:(g,h,x)\mapsto(v,w,x)$. We decompose the coordinate
transform $\gamma$ into two parts $\gamma=(\gamma_1,\gamma_2)$ where
$\gamma_1(g,h,x)=v$ is the $\mathbb R^d$-valued function and
$\gamma_2(g,h,x)=(w,x)$ is the $\mathbb R^{2d}$-valued function. The Jacobian
$\gamma'$ corresponding to the decomposition of the variables is given by 
\begin{eqnarray*}
\gamma'=\begin{bmatrix} \gamma'_{11} & \gamma'_{12}\\ \gamma'_{21} & \gamma'_{22}
\end{bmatrix}
=\begin{bmatrix} \partial_{g}\gamma_1 & \partial_{(h,x)}\gamma_1\\ \partial_{g}\gamma_2 &\partial_{(h,x)}\gamma_2
\end{bmatrix}.
\end{eqnarray*}
Finally, we get
\begin{align*}
B_2(g,h,x)=\frac1{(2\pi)^d}\int_{\mathbb R^d}e^{{\rm i}g\cdot\xi}\sigma_2(h,x,\xi)d\xi,
\end{align*}
where the principal symbol of $\sigma_2$ has the form
\begin{align}\label{eq:2p}
\sigma_2^p(h,x,\xi)=\phi\Big(\frac{w(0,h,x)}2\Big)\vartheta(x)\left|(\gamma_{11}
'(0,h,x))^{-\top}\xi\right|^{-m}\left|\det(\gamma_{11}'(0,h,
x))\right|^{-1}L^\tau(0,h,x).
\end{align}
Here $\alpha:=\frac{h_2}{h_1}$, $\beta:=\frac{h_3}{h_1}$,
\begin{equation*}
w(0,h,x)=\left\{\begin{aligned}
&2h_1(\sin\alpha,\cos\alpha)+2x,\quad d=2,\\
&2h_1(\sin\alpha\cos\beta,\sin\alpha\sin\beta,\cos\alpha)+2x,\quad d=3,
\end{aligned}\right.
\end{equation*}
\begin{equation*}
\gamma_{11}'(0,h,x)=\left\{\begin{aligned}
&2\left[
\begin{array}{cc}
\sin\alpha-\alpha\cos\alpha&\cos\alpha\\
\cos\alpha+\alpha\sin\alpha&-\sin\alpha
\end{array}
\right],\quad d=2,\\
&2\left[
\begin{array}{ccc}
\sin\alpha\cos\beta-\alpha\cos\alpha\cos\beta+\beta\sin\alpha\sin\beta&\cos\alpha\cos\beta&-\sin\alpha\sin\beta\\
\sin\alpha\sin\beta-\alpha\cos\alpha\sin\beta-\beta\sin\alpha\cos\beta&\cos\alpha\sin\beta&\sin\alpha\cos\beta\\
\cos\alpha+\alpha\sin\alpha&-\sin\alpha&0
\end{array}
\right],\quad d=3,
\end{aligned}\right.
\end{equation*}
and the residual $r_2:=\sigma_2-\sigma_2^p\in S_{1,0}^{-m-1}$.
Note that $B_2(g,h,x)=[\mathcal{F}^{-1}\sigma_2(h,x,\cdot)](g)$. Hence,
\begin{align*}
I_d(x,\omega_1,\omega_2)&=\int_{\mathbb R^d}\int_{\mathbb R^d}e^{{\rm
i}\left((c_1\omega_1+c_2\omega_2)g\cdot e_1+(c_1\omega_1-c_2\omega_2)h\cdot e_1\right)}[\mathcal{F}^{-1}\sigma_2(h,x,\cdot)](g)dgdh\\
&=\int_{\mathbb R^d}e^{{\rm
i}(c_1\omega_1-c_2\omega_2)h_1}\sigma_2(h,x,-(c_1\omega_1+c_2\omega_2)e_1)dh\\
&=\frac{1}{({\rm
i}(c_1\omega_1-c_2\omega_2))^M}\int_{\mathbb R^d}e^{{\rm
i}(c_1\omega_1-c_2\omega_2)h_1}\partial_{h_1}^M\sigma_2(h,x,-(c_1\omega_1+c_2\omega_2)e_1)dh,
\end{align*}
where $|\partial_{h_1}^M\sigma_2(h,x,-(c_1\omega_1+c_2\omega_2)e_1)|\le
C_M|c_1\omega_1+c_2\omega_2|^{-m}$. Consequently,
\[
|I_d(x,\omega_1,\omega_2)|\le C_M(1+|\omega_1-\omega_2|)^{-M}(\omega_1+\omega_2)^{-m},
\]
which completes the proof of \eqref{d23}.

For \eqref{d25}, letting $\omega_1=\omega_2=\omega$, we have from 
\eqref{eq:2p} that 
\begin{align*}
I_d(x,\omega,\omega)&=\int_{\mathbb R^d}e^{{\rm
i}(c_1-c_2)\omega h_1}\sigma_2(h,x,-(c_1+c_2)\omega e_1)dh\\
& =\int_{\mathbb R^d}e^{{\rm
i}(c_1-c_2)\omega h_1}\sigma_2^p(h,x,-(c_1+c_2)\omega e_1)dh+O(\omega^{-(m+1)})\\
& =\int_{\mathbb R^d}e^{{\rm
i}(c_1-c_2)\omega h_1}\phi\Big(\frac{w(0,h,x)}2\Big)\vartheta(x)\left|(\gamma_{11}
'(0,h,x))^{-\top}(c_1+c_2)\omega e_1\right|^{-m}\\
&\quad\times \left|\det(\gamma_{11}'(0,h,
x))\right|^{-1}L^\tau(0,h,x)dh+O(\omega^{-(m+1)}).
\end{align*}

If $d=2$,  by the expression of $\gamma_{11}'(0,h,x)$, we have for any $x\in U$
that 
\[
I_2(x,\omega,\omega)=\frac{2^{m-2}}{(c_1+c_2)^m\omega^{m}}\int_{\mathbb R^2}e^{{\rm
i}(c_1-c_2)\omega
h_1}\phi\Big(\frac{w(0,h,x)}2\Big)L^\tau(0,h,x)dh+O(\omega^{-(m+1)}),
\]
where
\[
L^\tau(0,h,x)=F_2(\tau^{-1}(0,h,x))|\det((\tau^{-1})'(0,h,x))|.
\]
Here
\begin{align*}
F_2(\tau^{-1}(0,h,x))=\frac{(-h_1\sin\alpha)^{d_{11}+d_{21}}(-h_1\cos\alpha)^{d_{12}+d_{22}}}{h_1^{d_1+d_2}}
\end{align*}
and $|\det((\tau^{-1})'(0,h,x))|=4$ are proved in \cite[Proposition 4.1]{LW}.
To simplify the expression of $I_2$, we consider another
coordinate transformation $\varrho:{\mathbb{R}}^2\to{\mathbb{R}}^2$ defined by 
\[
\varrho(h)=\zeta:=h_1(\sin\alpha,\cos\alpha)+x,
\]
which has the Jacobian
\[
\text{det}(\varrho')=\left|\begin{array}{cc}
\sin\alpha-\alpha\cos\alpha&\cos\alpha\\[2pt]
\cos\alpha+\alpha\sin\alpha&-\sin\alpha
\end{array}\right|=-1.
\]
Then for any $x\in U$, we have
\begin{align*}
I_2(x,\omega,\omega)&=\frac{2^{m}}{(c_1+c_2)^m\omega^{m}}\int_{\mathbb
R^2}e^{{\rm
i}(c_1-c_2)\omega h_1}\phi\Big(\frac{w(0,h,x)}2\Big)\\
&\quad\frac{(-h_1\sin\alpha)^{d_{11}+d_{21}}(-h_1\cos\alpha)^{d_{12}+d_{22}}}{
h_1^{d_1+d_2}}dh+O(\omega^{-(m+1)})\\
&=\frac{2^{m}}{(c_1+c_2)^m\omega^{m}}\int_{\mathbb R^2}e^{{\rm
i}(c_1-c_2)\omega|x-\zeta|}\phi(\zeta)\frac{\left(x_1-\zeta_1\right)^{d_{11}+d_{21}}\left(x_2-\zeta_2\right)^{d_{12}+d_{22}}}{|x-\zeta|^{d_1+d_2}}d\zeta
+O(\omega^{-(m+1)}).
\end{align*}

If $d=3$, then for any $x\in U$,
\[
I_3(x,\omega,\omega)=\frac{2^{m-3}}{(c_1+c_2)^m\omega^{m}}\int_{\mathbb R^3}e^{{\rm
i}(c_1-c_2)\omega h_1}\phi\Big(\frac{w(0,h,x)}2\Big)\frac{L^\tau(0,h,x)}{|\sin\alpha|}dh+O(\omega^{-(m+1)}),
\]
where
\[
L^\tau(0,h,x)=F_3(\tau^{-1}(0,h,x))|\det((\tau^{-1})'(0,h,x))|.
\]
Here
\begin{align*}
F_3(\tau^{-1}(0,h,x))=\frac{(-h_1\sin\alpha\cos\beta)^{d_{11}+d_{21}}(-h_1\sin\alpha\sin\beta)^{d_{12}+d_{22}}(-h_1\cos\alpha)^{d_{13}+d_{23}}}{h_1^{d_1+d_2}}
\end{align*}
and $|\det((\tau^{-1})'(0,h,x))|=8\sin^2\alpha$ are proved in \cite[Theorem
4.5]{LW}. By defining the transformation $\varrho:\mathbb R^3\to\mathbb R^3$
with
\[
\varrho(h)=\zeta:=h_1(\sin\alpha\cos\beta,\sin\alpha\sin\beta,\cos\alpha)+x,
\]
whose Jacobian satisfies
\[
|\det(\varrho^{-1})'|=\left|\det\left(\frac{\gamma_{11}'(0,h,x)}2\right)\right|^{-1}=\frac1{|\sin\alpha|},
\]
we finally get for any $x\in U$ that
\begin{align*}
I_3(x,\omega,\omega)&=\frac{2^{m-3}}{(c_1+c_2)^m\omega^{m}}\int_{\mathbb
R^3}e^{{\rm
i}(c_1-c_2)\omega h_1}\phi\Big(\frac{w(0,h,x)}2\Big)\frac{8\sin^2\alpha}{|\sin\alpha|}\\
&\quad\frac{(-h_1\sin\alpha\cos\beta)^{d_{11}+d_{21}}(-h_1\sin\alpha\sin\beta)^{
d_{12}+d_{22}}(-h_1\cos\alpha)^{d_{13}+d_{23}}}{h_1^{d_1+d_2}}dh+O(\omega^{
-(m+1)})\\
&=\frac{2^{m}}{(c_1+c_2)^m\omega^{m}}\int_{\mathbb R^3}e^{{\rm
i}(c_1-c_2)\omega|x-\zeta|}\phi(\zeta)\\
&\quad\frac{\left(x_1-\zeta_1\right)^{d_{11}+d_{21}}\left(x_2-\zeta_2\right)^{d_
{12}+d_{22}}\left(x_3-\zeta_3\right)^{d_{13}+d_{23}}}{|x-\zeta|^{d_1+d_2}}d\zeta
+O(\omega^{-(m+1)}),
\end{align*}
which completes the proof. 

\section{Proof of Lemma \ref{lm:v}}\label{app:2}

For simplicity, we first introduce the following notations $\bar{\sigma}:=-\frac{\rm i}{64}\left(\frac{2}{\pi}\right)^{\frac{3}{2}}$, $\bm{J}(\zeta):=\zeta\zeta^\top$,
\begin{eqnarray}\notag
\psi_1(\zeta):=c_{\rm s}^{-\frac{1}{2}}e^{{\rm i}\kappa_{\rm s}|\zeta|},\quad \psi_2(\zeta):=c_{\rm s}^{\frac{1}{2}}e^{{\rm i}\kappa_{\rm s}|\zeta|}-c_{\rm p}^{\frac{1}{2}}e^{{\rm i}\kappa_{\rm p}|\zeta|},\quad \psi_3(\zeta):=c_{\rm s}^{\frac{3}{2}}e^{{\rm i}\kappa_{\rm s}|\zeta|}-c_{\rm p}^{\frac{3}{2}}e^{{\rm i}\kappa_{\rm p}|\zeta|}
\end{eqnarray}
for $\zeta\in\mathbb R^2$, 
and the integral 
\begin{eqnarray*}%\label{e46}
&&A(\beta_1,\beta_2,\beta_3,p_1,p_2,p_3,\bm{M}_1,\bm{M}_2,\bm{M}_3):=
\int_D\int_D\frac{\beta_1(x-z)\beta_2(z-z')\beta_3(z'-x)}{|x-z|^{p_1}|z-z'|^{p_2
}|z'-x|^{p_3}}\times\\&&\qquad\qquad\qquad \qquad\qquad\qquad
\rho(z)\rho(z')\bm{M}_1(x-z)\bm{M}_2(z-z')\bm{M}_3(z'-x)\bm{a}dzdz',
\end{eqnarray*}
where $\beta_{i}\in\{\psi_1,\psi_2,\psi_3\}$ and $\bm{M}_i\in\{\bm{I},\bm{J}\}$ for $i=1,2,3$.
Substituting (\ref{d8}) into (\ref{eadd1}) shows that 
\begin{eqnarray*}%\label{h20}
\bm{v}(x,\omega,\bm{a})=\sum_{k=1}^4\bm{v}_k(x,\omega,\bm{a})\omega^{
-(k+\frac{1}{2})},
\end{eqnarray*}
where 
\begin{eqnarray*}%\label{e48}
\bm{v}_1(x,\omega,\bm{a})
=e^{-\frac{3}{4}\pi{\rm i}}\mu^{-3}\bar{\sigma} A(\psi_1,\psi_1,\psi_1,\frac{1}{2},\frac{1}{2},\frac{1}{2},\bm{I},\bm{I},\bm{I})+e^{-\frac{7}{4}\pi{\rm i}}\mu^{-2}\bar{\sigma} A(\psi_1,\psi_1,\psi_3,\frac{1}{2},\frac{1}{2},\frac{5}{2},\bm{I},\bm{I},\bm{J})\\
+e^{-\frac{7}{4}\pi{\rm i}}\mu^{-2}\bar{\sigma} A(\psi_1,\psi_3,\psi_1,\frac{1}{2},\frac{5}{2},\frac{1}{2},\bm{I},\bm{J},\bm{I})+e^{-\frac{11}{4}\pi{\rm i}}\mu^{-1}\bar{\sigma} A(\psi_1,\psi_3,\psi_3,\frac{1}{2},\frac{5}{2},\frac{5}{2},\bm{I},\bm{J},\bm{J})\\
+e^{-\frac{7}{4}\pi{\rm i}}\mu^{-2}\bar{\sigma} A(\psi_3,\psi_1,\psi_1,\frac{5}{2},\frac{1}{2},\frac{1}{2},\bm{J},\bm{I},\bm{I})+e^{-\frac{11}{4}\pi{\rm i}}\mu^{-1}\bar{\sigma} A(\psi_3,\psi_1,\psi_3,\frac{5}{2},\frac{1}{2},\frac{5}{2},\bm{J},\bm{I},\bm{J})\\
+e^{-\frac{11}{4}\pi{\rm i}}\mu^{-1}\bar{\sigma} A(\psi_3,\psi_3,\psi_1,\frac{5}{2},\frac{5}{2},\frac{1}{2},\bm{J},\bm{J},\bm{I})+e^{-\frac{15}{4}\pi{\rm i}}\bar{\sigma} A(\psi_3,\psi_3,\psi_3,\frac{5}{2},\frac{5}{2},\frac{5}{2},\bm{J},\bm{J},\bm{J}),
\end{eqnarray*}
\begin{eqnarray*}%\label{e49}
\bm{v}_2(x,\omega,\bm{a})
=e^{-\frac{5}{4}\pi{\rm i}}\mu^{-2}\bar{\sigma} A(\psi_1,\psi_1,\psi_2,\frac{1}{2},\frac{1}{2},\frac{3}{2},\bm{I},\bm{I},\bm{I})+e^{-\frac{9}{4}\pi{\rm i}}\mu^{-1}\bar{\sigma} A(\psi_1,\psi_3,\psi_2,\frac{1}{2},\frac{5}{2},\frac{3}{2},\bm{I},\bm{J},\bm{I})\\
+e^{-\frac{9}{4}\pi{\rm i}}\mu^{-1}\bar{\sigma} A(\psi_3,\psi_1,\psi_2,\frac{5}{2},\frac{1}{2},\frac{3}{2},\bm{J},\bm{I},\bm{I})+e^{-\frac{13}{4}\pi{\rm i}}\bar{\sigma} A(\psi_3,\psi_3,\psi_2,\frac{5}{2},\frac{5}{2},\frac{3}{2},\bm{J},\bm{J},\bm{I})\\
+e^{-\frac{5}{4}\pi{\rm i}}\mu^{-2}\bar{\sigma} A(\psi_1,\psi_2,\psi_1,\frac{1}{2},\frac{3}{2},\frac{1}{2},\bm{I},\bm{I},\bm{I})+e^{-\frac{9}{4}\pi{\rm i}}\mu^{-1}\bar{\sigma} A(\psi_1,\psi_2,\psi_3,\frac{1}{2},\frac{3}{2},\frac{5}{2},\bm{I},\bm{I},\bm{J})\\
+e^{-\frac{9}{4}\pi{\rm i}}\mu^{-1}\bar{\sigma} A(\psi_3,\psi_2,\psi_1,\frac{5}{2},\frac{3}{2},\frac{1}{2},\bm{J},\bm{I},\bm{I})+e^{-\frac{13}{4}\pi{\rm i}}\bar{\sigma} A(\psi_3,\psi_2,\psi_3,\frac{5}{2},\frac{3}{2},\frac{5}{2},\bm{J},\bm{I},\bm{J})\\
+e^{-\frac{5}{4}\pi{\rm i}}\mu^{-2}\bar{\sigma} A(\psi_2,\psi_1,\psi_1,\frac{3}{2},\frac{1}{2},\frac{1}{2},\bm{I},\bm{I},\bm{I})+e^{-\frac{9}{4}\pi{\rm i}}\mu^{-1}\bar{\sigma} A(\psi_2,\psi_1,\psi_3,\frac{3}{2},\frac{1}{2},\frac{5}{2},\bm{I},\bm{I},\bm{J})\\
+e^{-\frac{9}{4}\pi{\rm i}}\mu^{-1}\bar{\sigma} A(\psi_2,\psi_3,\psi_1,\frac{3}{2},\frac{5}{2},\frac{1}{2},\bm{I},\bm{J},\bm{I})+e^{-\frac{13}{4}\pi{\rm i}}\bar{\sigma} A(\psi_2,\psi_3,\psi_3,\frac{3}{2},\frac{5}{2},\frac{5}{2},\bm{I},\bm{J},\bm{J}),
\end{eqnarray*}
\begin{eqnarray*}%\label{e50}
\bm{v}_3(x,\omega,\bm{a})
=e^{-\frac{7}{4}\pi{\rm i}}\mu^{-1}\bar{\sigma} A(\psi_1,\psi_2,\psi_2,\frac{1}{2},\frac{3}{2},\frac{3}{2},\bm{I},\bm{I},\bm{I})+e^{-\frac{11}{4}\pi{\rm i}}\bar{\sigma} A(\psi_3,\psi_2,\psi_2,\frac{5}{2},\frac{3}{2},\frac{3}{2},\bm{J},\bm{I},\bm{I})\\
+e^{-\frac{7}{4}\pi{\rm i}}\mu^{-1}\bar{\sigma} A(\psi_2,\psi_1,\psi_2,\frac{3}{2},\frac{1}{2},\frac{3}{2},\bm{I},\bm{I},\bm{I})+e^{-\frac{11}{4}\pi{\rm i}}\bar{\sigma} A(\psi_2,\psi_3,\psi_2,\frac{3}{2},\frac{5}{2},\frac{3}{2},\bm{I},\bm{J},\bm{I})\\
+e^{-\frac{7}{4}\pi{\rm i}}\mu^{-1}\bar{\sigma} A(\psi_2,\psi_2,\psi_1,\frac{3}{2},\frac{3}{2},\frac{1}{2},\bm{I},\bm{I},\bm{I})+e^{-\frac{11}{4}\pi{\rm i}}\bar{\sigma} A(\psi_2,\psi_2,\psi_3,\frac{3}{2},\frac{3}{2},\frac{5}{2},\bm{I},\bm{I},\bm{J})
\end{eqnarray*}
and
\begin{eqnarray*}%\label{e51}
\bm{v}_4(x,\omega,\bm{a})
=e^{-\frac{9}{4}\pi{\rm i}}\bar{\sigma} A(\psi_2,\psi_2,\psi_2,\frac{3}{2},\frac{3}{2},\frac{3}{2},\bm{I},\bm{I},\bm{I}).
\end{eqnarray*}
Then applying the Cauchy--Schwartz inequality leads to
\begin{eqnarray}\label{e40}
\frac{1}{Q-1}\int_1^Q\omega^{m+2}|\bm v(x,\omega,\bm{a})|^2d\omega\lesssim
\frac{1}{Q-1}\sum_{k=1}^4\int_1^Q\omega^{m-(2k-1)}|\bm{v}_k(x,\omega,\bm
{a})|^2d\omega.
\end{eqnarray}
Noting that $\omega^{-\frac12}\frac{\beta_2(z-z')}{|z-z'|^{p_2}}\bm M_2(z-z')$ involved in $\bm v_k$ has the same singularity as $\bm G^{(0)}(z,z',\omega)$, we get
\begin{align*}
\left|\omega^{-\frac32}\bm v_k(x,\omega,\bm a)\right|&\lesssim\bigg|\omega^{-\frac32}\int_D\int_D\frac{\beta_1(x-z)\beta_2(z-z')\beta_3(z'-x)}{|x-z|^{p_1}|z-z'|^{p_2
}|z'-x|^{p_3}}\\&\quad\times
\rho(z)\rho(z')\bm{M}_1(x-z)\bm{M}_2(z-z')\bm{M}_3(z'-x)\bm{a}dzdz'\bigg|\\
&\lesssim\left\|\omega^{-1}\rho\otimes\rho\frac{\beta_1(x-\cdot)\beta_3(\cdot-x)}{|x-\cdot|^{p_1}|\cdot-x|^{p_3}}\bm M_1(x-\cdot)\bm M_3(\cdot-x)\bm a\right\|_{\bm W_0^{-2s,q}(D\times D)}\\
&\quad\times\left\|\omega^{-\frac12}\frac{\beta_2(\cdot-\cdot)}{|\cdot-\cdot|^{p_2}}\bm M_2(\cdot-\cdot)\right\|_{(W^{2s,p}(D\times D))^{2\times2}}\\
&\lesssim\omega^{-\frac32+3s}
\end{align*}
based on Lemma \ref{lm:G0} and a similar argument used in \eqref{eq:u-v}. As a result,
\begin{align}\label{eq:vk2}
\lim_{Q\to\infty}\frac{1}{Q-1}\sum_{k=2}^4\int_1^Q\omega^{m-(2k-1)}|\bm{v}_k(x,\omega,\bm
{a})|^2d\omega
&\lesssim\lim_{Q\to\infty}\frac{1}{Q-1}\sum_{k=2}^4\int_1^Q\omega^{m-(2k-1)+6s}d\omega\notag\\
&\lesssim\lim_{Q\to\infty}\frac{Q^{m-2+6s}-1}{Q-1}=0
\end{align}
almost surely by choosing $s\in(\frac{2-m}2,\frac12)$ with $m\in(\frac53,2]$ such that $m-2+6s<1$. 
Hence, to prove \eqref{eq:v}, according to \eqref{e40} and \eqref{eq:vk2}, we only need to show for $x\in U$ that 
\begin{eqnarray}\label{h24}
\lim_{Q\to\infty}\frac{1}{Q-1}\int_1^Q\omega^{m-1}|\bm{v}_1(x,\omega,\bm{a}
)|^2d\omega = 0
\end{eqnarray}
in the almost surely sense. 
Note that
\begin{eqnarray*}\label{e42}
\frac{1}{Q-1}\int_1^Q\omega^{m-1}|\bm{v}_1(x,\omega,\bm{a})|^2d\omega\leq
\int_1^{\infty}\frac{\omega\bm 1_{[1,Q]}(\omega)}{Q-1}\omega^{m-2}|\bm{v}_1(x,\omega,\bm{a}
)|^2d\omega
\end{eqnarray*}
with $\bm 1_{[1,Q]}$ being the characteristic function on the interval $[1,Q]$, and it holds point-wisely that
\[
\lim_{Q\to\infty}\frac{\omega\bm 1_{[1,Q]}(\omega)}{Q-1}=0\quad\text{and}\quad\left|\frac{\omega\bm 1_{[1,Q]}(\omega)}{Q-1}\omega^{m-2}|\bm{v}_1(x,\omega,\bm{a}
)|^2\right|\lesssim\omega^{m-2}|\bm v_1(x,\omega,\bm{a}
)|^2.
\]
By the dominated convergence theorem, to show \eqref{h24}, it suffices to prove 
\begin{eqnarray}\label{e39}
\int_1^{\infty}\omega^{m-2}\mathbb E|\bm{v}_1(x,\omega,\bm{a})|^2d\omega<\infty.
\end{eqnarray}
Substituting $\psi_1$, $\psi_2$, $\psi_3$ and $\bm{J}$ into
$\bm{v}_1$ gives that $\bm{v}_1$ is a linear combination of the integral
\begin{eqnarray}\label{e45}
{\mathbb B}(x,\omega):=\int_D\int_De^{{\rm i}\omega(c_1|x-z|+c_2|z-z'|+c_3|z'-x|)}{\mathbb K}(x,z,z')\rho(z)\rho(z')dzdz',
\end{eqnarray}
where $c_1,c_2,c_3\in\{c_{\rm s},c_{\rm p}\}$ and 
\begin{eqnarray}\label{e46}
{\mathbb K}(x,z,z')=\frac{(x_1-z_1)^{p_1}(x_2-z_2)^{p_2}(z_1-z'_1)^{p_3}(z_2-z'_2)^{p_4}(z'_1-x_1)^{p_5}(z'_2-x_2)^{p_6}}{|x-z|^{p_7}|z-z'|^{p_8}|z'-x|^{p_9}}
\end{eqnarray}
with 
\begin{align*}
(p_1,\cdots,p_9)\in\Big\{(p_1,\cdots,p_9)\Big|& p_i\in\{0,1,2\},1\le i\le6, p_j\in\Big\{\frac12,\frac52\Big\},7\le j\le9, \\
&p_7-p_1-p_2=\frac{1}{2},p_8-p_3-p_4=\frac{1}{2},p_9-p_5-p_6=\frac{1}{2}\Big\}.
\end{align*}

It follows from
the Cauchy--Schwartz inequality that a sufficient condition for (\ref{e39}) is 
\begin{eqnarray}\label{e50}
\int_1^{\infty}\omega^{m-2}\mathbb E|\mathbb B(x,\omega)|^2d\omega<\infty.
\end{eqnarray}
To deal with the roughness of the random potential $\rho$, similar to the technique used in \cite{LLW}, we introduce a modification $\rho_\varepsilon:=\rho*\varphi_\varepsilon$ with $\varphi_\varepsilon(x):=\varepsilon^{-2}\varphi(x/\varepsilon)$ for $\varepsilon>0$, where $\varphi\in C_0^\infty(\mathbb R^2)$ is a radially symmetric function satisfying $\int_{\mathbb R^2}\varphi(x)dx=1$, and define
\begin{align}\label{eq:Bepsilon}
{\mathbb B}_\varepsilon(x,\omega):=\int_D\int_De^{{\rm i}\omega(c_1|x-z|+c_2|z-z'|+c_3|z'-x|)}{\mathbb K}(x,z,z')\rho_\varepsilon(z)\rho_\varepsilon(z')dzdz'
\end{align}
by replacing $\rho$ in \eqref{e45} by $\rho_\varepsilon$.
It is easy to show that $\lim_{\varepsilon\to\infty}\mathbb E|\mathbb B_\varepsilon|^2=\mathbb E|\mathbb B|^2$ (cf. \cite{LLW}), which, together with Fatou's lemma and the fact $m\in(\frac53,2]$, leads to
\[
\int_1^{\infty}\omega^{m-2}\mathbb E|\mathbb B(x,\omega)|^2d\omega\le\varlimsup_{\varepsilon\to 0}\int_1^{\infty}\mathbb E|\mathbb B_\varepsilon(x,\omega)|^2d\omega.
\]
As a result, it suffices to show
\begin{eqnarray}\label{e51}
\varlimsup_{\varepsilon\to 0}\int_1^{\infty}{\mathbb E}|{\mathbb B}_{\varepsilon}(x,\omega)|^2d\omega<\infty\quad \forall~x\in U.
\end{eqnarray}

The procedure to show \eqref{e51} is similar to the proof for the acoustic wave case with $m=d=2$ given in \cite{LPS08}. For the self-contained purpose and completeness, we present the details below. 

The basic idea is to express ${\mathbb B}_{\varepsilon}$ in terms of a one-dimensional Fourier transform and then get the estimate with respect to $\omega$ by utilizing the Parseval formula. 
To this end, we first consider the phase function $L(z,z'):=c_1|x-z|+c_2|z-z'|+c_3|z'-x|$, which is smooth in the domain $\Theta:=\{(z,z')\in D\times D| z\neq z'\}$ and
\begin{eqnarray*}\label{e55}
\nabla_z L(z,z') = c_1\frac{z-x}{|z-x|}+c_2\frac{z-z'}{|z-z'|},\quad \nabla_{z'} L(z,z') = c_2\frac{z'-z}{|z'-z|}+c_3\frac{z'-x}{|z'-x|}\quad\forall~(z,z')\in\Theta.
\end{eqnarray*}
Without loss of generality, we assume that $0\in U$ such that  $|z|$ and $|z'|$ are bounded from below and above for $z,z'\in D$ since $U$ has a positive distance to $D$.  
Hence, it holds for $(z,z')\in \Theta$ that
\begin{eqnarray}\label{e56}
0<C_1\le |\nabla L(z,z')|\le C_2<\infty
\end{eqnarray}
for some constants $C_1$ and $C_2$, where we use the facts that $U$ is bounded and convex, and that
\begin{eqnarray*}\label{e58}
(z,z')\cdot \nabla L(z,z')&=&c_1z\cdot\frac{z-x}{|z-x|}+c_2|z-z'|+c_3z'\cdot\frac{z'-x}{|z'-x|}\notag\\
&=&c_1|z|\cos\theta_1+c_2|z-z'|+c_3|z'|\cos\theta_2
\ge C_3>0
\end{eqnarray*}
for some constant $C_3$ with $\theta_1$ and $\theta_2$ being the angle between $z$ and $z-x$ and the angle between $z'$ and $z'-x$, respectively. 
Due to the boundedness of $D$ and $U$ and the fact that they have a positive distance, the surface
\begin{eqnarray*}\label{e59}
\Gamma_t':=\{(z,z')\in D\times D|L(z,z')=t\},\quad  t>0
\end{eqnarray*}
is nonempty only for $t\in[T_0,T_1]$ with some positive values $T_0=T_0(x)$ and $T_1=T_1(x)$.

For a fixed $\tilde{t}\in [T_0, T_1]$, there exists a $\tilde{\eta}=\tilde{\eta}(\tilde{t})$ and an open cone $E=E(\tilde{t})\subset {\mathbb R^4}$ centered at the origin such that it holds for $t_0=t_0(\tilde t):=\tilde{t}-\tilde{\eta}$ and $t_1=t_1(\tilde t):=\tilde{t}+\tilde{\eta}$ that
\begin{eqnarray}\label{e60}
D\times D\cap\{t_0<L(z,z')<t_1\}\subset E\cap\{t_0<L(z,z')<t_1\}:=\Gamma
\end{eqnarray}
and
\[
\Gamma=\bigcup_{t\in[t_0,t_1]}\Gamma_t\quad\text{with}\quad\Gamma_t:=\Gamma\cap\{(z,z')|L(z,z')=t\}.
\]

According to (\ref{eq:Bepsilon}) and (\ref{e60}), we obtain 
\begin{align*}%\label{e67}
{\mathbb B}_{\varepsilon}(x,\omega)&=\int_{\Gamma}e^{{\rm i}\omega L(z,z')}{\mathbb K}(x,z,z')\rho_{\varepsilon}(z)\rho_{\varepsilon}(z')dzdz'\\
&=\int_{t_0}^{t_1}e^{{\rm i}\omega t}\int_{\Gamma_t}{\mathbb K}(x,z,z')|\nabla L(z,z')|^{-1}\rho_{\varepsilon}(z)\rho_{\varepsilon}(z')d\mathcal{H}^3(z,z')dt\\
&=:\int_{t_0}^{t_1}e^{{\rm i}\omega t}S_{\varepsilon}(t)dt=(\mathcal{F}S_{\varepsilon})(-\omega),
\end{align*}
where 
\begin{eqnarray}\label{e68}
S_{\varepsilon}(t):=\int_{\Gamma_t}{\mathbb K}(x,z,z')|\nabla L(z,z')|^{-1}\rho_{\varepsilon}(z)\rho_{\varepsilon}(z')d\mathcal{H}^3(z,z')
\end{eqnarray}
is compactly supported in $[T_0, T_1]$ and the integral in \eqref{e68} is with respect to the three-dimensional Hausdorff measure $\mathcal{H}^3$ on $\Gamma_t$. 
Note that
\begin{eqnarray}\notag
{\mathbb E}|S_{\varepsilon}(t)|^2=\int_{\Gamma_t\times \Gamma_t}{\mathbb K}(x,z,z'){\mathbb K}(x,\tilde{z},\tilde{z}')|\nabla L(z,z')|^{-1}|\nabla L(\tilde{z},\tilde{z}')|^{-1}\\\label{e72}
{\mathbb E}[\rho_{\varepsilon}(z)\rho_{\epsilon}(z')\rho_{\varepsilon}(\tilde{z})\rho_{\varepsilon}(\tilde{z}')]d\mathcal{H}^3(z,z')d\mathcal{H}^3(\tilde{z},\tilde{z}'),
\end{eqnarray}
where $|\nabla L(z,z')|^{-1}|\nabla L(\tilde{z},\tilde{z}')|^{-1}$ is bounded according to \eqref{e56} and 
\begin{eqnarray*}\label{e73}
|{\mathbb K}(x,z,z')|\leq |x-z|^{-\frac{1}{2}}|z-z'|^{-\frac{1}{2}}|z'-x|^{-\frac{1}{2}}\lesssim |z-z'|^{-\frac{1}{2}}
\end{eqnarray*}
for $x\in U$ and $(z,z')\in\Theta$ according to (\ref{e46}).
Moreover, the Wick formula leads to
\begin{eqnarray*}\label{e75}
\mathbb E[\rho_{\varepsilon}(z)\rho_{\varepsilon}(z')\rho_{\varepsilon}(\tilde{z})\rho_{\varepsilon}(\tilde{z}')]
=\mathcal K_{\rho_\varepsilon}(z,z')\mathcal K_{\rho_\varepsilon}(\tilde{z},\tilde{z}')+\mathcal K_{\rho_\varepsilon}(z,\tilde{z})\mathcal K_{\rho_\varepsilon}(z',\tilde{z}')+\mathcal K_{\rho_\varepsilon}(z,\tilde{z}')\mathcal K_{\rho_\varepsilon}(z',\tilde{z}),
\end{eqnarray*}
where $\mathcal K_{\rho_\varepsilon}(z,z'):=\mathbb E[\rho_\varepsilon(z)\rho_\varepsilon(z')]$ is the covariance function of $\rho_\varepsilon$.
Thus, (\ref{e72}) turns to be
\begin{align*}
{\mathbb E}|S_{\varepsilon}(t)|^2&\lesssim\int_{\Gamma_t\times\Gamma_t}|z-z'|^{-\frac{1}{2}}|\tilde{z}-\tilde{z}'|^{-\frac{1}{2}}\left|\mathcal K_{\rho_\varepsilon}(z,z')\mathcal K_{\rho_\varepsilon}(\tilde{z},\tilde{z}')\right|d\mathcal{H}^3(z,z')d\mathcal{H}^3(\tilde{z},\tilde{z}')\\
&\quad +\int_{\Gamma_t\times\Gamma_t}|z-z'|^{-\frac{1}{2}}|\tilde{z}-\tilde{z}'|^{-\frac{1}{2}}\left|\mathcal K_{\rho_\varepsilon}(z,\tilde{z})\mathcal K_{\rho_\varepsilon}(z',\tilde{z}')\right|d\mathcal{H}^3(z,z')d\mathcal{H}^3(\tilde{z},\tilde{z}')\\
&\quad +\int_{\Gamma_t\times\Gamma_t}|z-z'|^{-\frac{1}{2}}|\tilde{z}-\tilde{z}'|^{-\frac{1}{2}}\left|\mathcal K_{\rho_\varepsilon}(z,\tilde{z}')\mathcal K_{\rho_\varepsilon}(z',\tilde{z})\right|d\mathcal{H}^3(z,z')d\mathcal{H}^3(\tilde{z},\tilde{z}')\\
&=:{\rm I}_1^*+{\rm I}_2^*+{\rm I}^*_3.
\end{align*}
For sufficiently small $\varepsilon>0$, it follows from \cite[Lemma 10]{LLW} that
\begin{eqnarray*}\label{e76}
&&|\mathcal K_{\rho_\varepsilon}(z,z')|\lesssim|\ln|z-z'||+O(1)\quad{\rm for}\;\; m=2,\\
&&|\mathcal K_{\rho_\varepsilon}(z,z')|\lesssim|z-z'|^{-(2-m)}+O(1)\quad{\rm for}\;\; m\in (1, 2).
\end{eqnarray*}
Hence, for any $m\in(\frac53,2]$, there exists a sufficiently small $\epsilon>0$ such that
\[
|\mathcal K_{\rho_\varepsilon}(z,z')|\lesssim|z-z'|^{-(2-m+\epsilon)}
\]
when $|z-z'|\ll1$.

For ${\rm I}^*_1$, we have 
\begin{eqnarray*}\label{e78}
{\rm I}^*_1\lesssim\int_{\Gamma_t}|z-z'|^{-\frac{1}{2}}|\mathcal K_{\rho_\varepsilon}(z,z')|d\mathcal{H}^3(z,z')\int_{
\Gamma_t
}|\tilde{z}-\tilde{z}'|^{-\frac{1}{2}}|\mathcal K_{\rho_\varepsilon}(\tilde z,\tilde z')|d\mathcal{H}^3(\tilde{z},\tilde{z}
')<\infty
\end{eqnarray*}
according to \cite[Lemma 6]{LPS08}. 

For ${\rm I}^*_2$, it follows from the H\"{o}lder inequality and \cite[Lemma 6]{LPS08} that 
\begin{align*}%\label{e79}
{\rm I}^*_2&\lesssim\int_{\Gamma_t\times\Gamma_t}|z-z'|^{-\frac{1}{2}}|\tilde{z}-\tilde{z}'|^{-\frac{1}{2}}|z-\tilde{z}|^{-(2-m+\epsilon)}|z'-\tilde{z}'|^{-(2-m+\epsilon)}d\mathcal{H}^3(z,z')d\mathcal{H}^3(\tilde{z},\tilde{z}')\\
&\lesssim \left[\int_{\Gamma_t\times\Gamma_t}|z-z'|^{-\frac{3}{2}}|\tilde{z}-\tilde{z}'|^{-\frac{3}{2}}d\mathcal{H}^3(z,z')d\mathcal{H}^3(\tilde{z},\tilde{z}')\right]^{\frac{1}{3}}\times\\
&\quad \left[\int_{\Gamma_t\times\Gamma_t}|z-\tilde{z}|^{-\frac{3}{2}(2-m+\epsilon)}|z'-\tilde{z}'|^{-\frac{3}{2}(2-m+\epsilon)}d\mathcal{H}^3(z,z')d\mathcal{H}^3(\tilde{z},\tilde{z}')\right]^{\frac{2}{3}}\\
&\lesssim \left[\int_{\Gamma_t}|z-z'|^{-\frac{3}{2}}d\mathcal{H}^3(z,z')\int_{\Gamma_t}|\tilde{z}-\tilde{z}'|^{-\frac{3}{2}}d\mathcal{H}^3(\tilde{z},\tilde{z}')\right]^{\frac{1}{3}}\times\\
&\quad\left[\int_{\Gamma_t\times\Gamma_t}|z-\tilde{z}|^{-3(2-m+\epsilon)}d\mathcal{H}^3(z,z')d\mathcal{H}^3(\tilde{z},\tilde{z}')\right]^{\frac13}\times\\
&\quad\left[\int_{\Gamma_t\times\Gamma_t}|z'-\tilde{z}'|^{-3(2-m+\epsilon)}d\mathcal{H}^3(z,z')d\mathcal{H}^3(\tilde{z},\tilde{z}')\right]^{\frac{1}{3}}<\infty
\end{align*}
for $m\in (\frac53, 2]$. 
An argument similar to the one used in ${\rm I}^*_2$ shows that ${\rm I}^*_3<\infty$. 

Hence, for any fixed $\tilde t\in[T_0,T_1]$, there exists a constant $C(\tilde t)$ such that
\begin{eqnarray*}\label{e69}
{\mathbb E}|S_{\varepsilon}(t)|^2\leq C(\tilde{t})\quad {\rm for\;\;all\;\;} t\in (t_0(\tilde{t}), t_1(\tilde{t}))
\end{eqnarray*}
and sufficiently small $\varepsilon>0$.
By compactness, there is a countable subset $ \Lambda\subset[T_0,T_1]$ with finite elements such that
\[
[T_0,T_1]\subset\bigcup_{\tilde t\in\Lambda}(t_0(\tilde t),t_1(\tilde t)).
\]
By defining $C:=\sum_{\tilde t\in\Lambda}C(\tilde t)$, we get 
\begin{eqnarray*}\label{e70}
{\mathbb E}|S_{\varepsilon}(t)|^2\leq C\quad {\rm for\;\;all\;\;} t\in [T_0, T_1]
\end{eqnarray*}
and sufficiently small $\varepsilon>0$. Then it follows from the Parseval formula that
 \begin{eqnarray*}\label{e71}
\varlimsup_{\varepsilon\to 0}\int_1^{\infty}{\mathbb E}|{\mathbb B}_{\varepsilon}(x,\omega)|^2d\omega=\varlimsup_{\varepsilon\to 0}\int_{T_0}^{T_1}{\mathbb E}|S_{\varepsilon}(t)|^2dt\leq C(T_1-T_0)<\infty,
 \end{eqnarray*}
which yields \eqref{e51} and thus \eqref{e50}. 
Then \eqref{e39} holds due to \eqref{e50}, which completes the proof together with \eqref{eq:vk2}.

\end{document}